\documentclass[12pt]{elsarticle}
\makeatletter
\def\ps@pprintTitle{%
 \let\@oddhead\@empty
 \let\@evenhead\@empty
 \def\@oddfoot{}%
 \let\@evenfoot\@oddfoot}
\makeatother

\usepackage{hyperref}
\usepackage{amssymb}
\usepackage{latexsym}
\usepackage{amsmath,amsfonts,amssymb}
\usepackage{amsthm}

\usepackage{tikz}
\usetikzlibrary{decorations.pathreplacing}
\usetikzlibrary{arrows}

\newtheorem{theorem}{Theorem}[section]
\newtheorem{lemma}[theorem]{Lemma}
\newtheorem{proposition}[theorem]{Proposition}
\newtheorem{corollary}[theorem]{Corollary}

\theoremstyle{definition}
\newtheorem{definition}[theorem]{Definition}

\newtheorem{example}[theorem]{Example}
\newtheorem{remark}[theorem]{Remark}

\def\DM{\operatorname{DM}}

\def\Nat{{\mathbb N}}

\def\wiel{\operatorname{Wi}}
\def\bzero{{\mathbf 0}}
\def\bunity{{\mathbf 1}}
\def\0{{\mathbf 0}}
\def\1{{\mathbf 1}}

\def\N{{\mathbb N}}
\def\R{{\mathbb R}}
\def\Rmax{\R_{\max}}
\def\Rp{\Rmax}
\def\Rpnn{\Rp^{n\times n}}

\def\digr{{\mathcal D}}

\def\crit{{\mathcal G}^c}

\def\subcrit{{\mathcal G}}
\def\mN{\mathcal{N}}
\def\Z{\mathbb{Z}}

\def\Wi{\operatorname{Wi}}

\newcommand{\walkslen}[3]{\mathcal{W}^{#3}(#1\to #2)}

\newcommand{\walkslennode}[4]{\mathcal{W}^{#3}(#1\xrightarrow{#4} #2)}

\newcommand{\walksnode}[3]{\mathcal{W}(#1\xrightarrow{#3} #2)}

\newcommand{\walks}{\mathcal{W}}

\newcommand{\trans}[1]{#1^T}

\newcommand{\old}[1]{}

\def\bnacht{B_{\operatorname{N}}}

\def\g{\operatorname{g}}

\colorlet{darkgreen}{black!50!green}

\begin{document}

%
\setcounter{footnote}{1}

\begin{frontmatter}
\setcounter{footnote}{1}
\title{On the Tightness of Bounds for Transients of Weak CSR Expansions and Periodicity Transients of 
Critical Rows and Columns of Tropical Matrix Powers\tnoteref{t1}}
\tnotetext[t1]{This work was partially supported by ANR
Perturbations grant (ANR-10-BLAN 0106). 
The work of S. Sergeev was also supported by EPSRC grant EP/P019676/1.}

\author[rvt1]{Glenn Merlet}
\ead{glenn.merlet@univ-amu.fr}

\author[rvt2]{Thomas Nowak\fnref{fntn}}
\ead{thomas.nowak@lri.fr}

\author[rvt3]{Serge{\u\i} Sergeev\corref{cor}}
\ead{sergiej@gmail.com}

\address[rvt1]{Aix Marseille Univ, CNRS, Centrale Marseille, I2M, Marseille, France}
\address[rvt2]{Universit\'e Paris-Saclay, CNRS, Orsay, France}
\address[rvt3]{University of Birmingham, School of Mathematics, 
Edgbaston B15 2TT, UK.}

\cortext[cor]{Corresponding author. Email: sergiej@gmail.com}

\fntext[fntn]{T. Nowak was with the {\'E}cole Normale Sup\'erieure, Paris,
France, when this work was initiated.}

\begin{abstract}
We study the transients of matrices in max-plus algebra.
Our approach is based on the weak CSR expansion.
Using this expansion, the transient can be expressed by 
$\max\{T_1,T_2\}$, where~$T_1$ is the weak CSR threshold
and~$T_2$ is the time after which the purely pseudoperiodic CSR terms start to
dominate in the expansion. 
Various bounds have been derived for~$T_1$ and~$T_2$, naturally leading to the 
question which matrices, if any, attain these bounds.

In the present paper we characterize the matrices attaining two particular
bounds on~$T_1$, which are
generalizations
of the bounds of Wielandt and Dulmage-Mendelsohn on the indices of
non-weighted digraphs.
This also leads to a characterization of tightness for the same bounds on the
transients of critical rows and columns.
The characterizations themselves are generalizations of those
for the non-weighted case.
\end{abstract}

\end{frontmatter}


\section{Introduction}

Max-plus algebra is a version of linear algebra developed over the max-plus
semiring, which is the set $\Rmax=\R\cup\{-\infty\}$
equipped with 
the operations $a\oplus b:=\max\{a,b\}$ (additive) and $a\otimes b:=a+b$ 
(multiplicative). 
This semiring has a zero $\bzero:=-\infty$, neutral with respect to 
$\oplus$, and a unity $\bunity=0$, neutral with respect to $\otimes$. The multiplicative
operation is invertible, that is, for each $\alpha\neq\bzero$ there exists an element 
$\alpha^-=-\alpha$ such that $\alpha^-\otimes\alpha=\alpha\otimes \alpha^-=\bunity$. 

These arithmetical operations are extended to matrices and vectors in the usual way.
Matrix addition is defined by $(A\oplus B)_{i,j}=a_{i,j}\oplus b_{i,j}$ for two matrices 
$A=(a_{i,j})$ and $B=(b_{i,j})$ of equal dimensions, and matrix multiplication by 
$(A\otimes B)_{i,j}=\bigoplus_{k=1}^l a_{ik}\otimes b_{kj}$ for two matrices $A$ and $B$ of compatible dimensions. Here we are  
interested in tropical matrix powers:
\begin{equation}
\label{e:tropicalpowers}
A^t=\overbrace{A\otimes A\otimes A\cdots\otimes A}^{t\text{ times}},\quad t\geq 1,
\enspace.
\end{equation}
assuming that $A^0=I$, the max-plus identity matrix, in which all diagonal entries are equal to $\bunity=0$
and all off-diagonal entries are equal to $\bzero=-\infty$.

As we see, matrix powers are easy to define for natural $t$. As for negative $t$, the problem is 
that the set of max-plus matrices for which inverse exists is very scarce (see, e.g.,\cite{But} Theorem 1.1.3
for a complete description).  However, we will make use of invertible max-plus diagonal matrices: matrices
$D=(d_{i,j})$ in which $d_{i,i}$ are real for all $i$ and $d_{i,j}=-\infty$ when $i\neq j$. The inverse of 
$D$, denoted by $D^-$, is also a diagonal matrix with diagonal entries equal to $d_{i,i}^-$ for all $i$, 
so that we have $D\otimes D^-=D^-\otimes D=I$.  

In what follows, the multiplication sign $\otimes$ will be always omitted in the case 
of matrix multiplication, but always kept in the case of multiplication by scalars. In particular, we 
write $\lambda^{\otimes t}=\overbrace{\lambda\otimes \ldots\otimes
\lambda}^{t\text{ times}}=t\lambda$
and $\lambda^{\otimes 1/t}=\frac1t\lambda$ for $\lambda\in\Rmax$.

The fundamental result on tropical matrix powers~\cite{Cohen+} states
that if $A$ is irreducible 
then
there exist a real $\lambda$ and integers $\gamma$ and $T$
such that
\begin{equation}
\label{e:period} 
\forall t\geq T\colon\quad
A^{t+\gamma}=\lambda^{\otimes \gamma}\otimes A^t.
\end{equation}
The smallest nonnegative~$T$ for which~\eqref{e:period}
holds is called the {\em transient } of matrix~$A$, 
denoted by $T(A)$. 
The transient can be shown to be independent of the choice of~$\lambda$
and~$\gamma$.
In fact, $\lambda$ is the largest mean cycle weight in the weighted digraph
described by~$A$.
Bounds on
transients have been studied by many authors, e.g., Hartmann and
Arguelles~\cite{HA-99}, Bouillard and Gaujal~\cite{BG-00}
Soto~y Koelemeijer~\cite{SyK:03},  Akian et al.~\cite{AGW-05},
Charron-Bost et al.~\cite{CBFN-17}, and the authors~\cite{MNS}.
The time behavior and complexity of several real systems can be reduced to
determining the transient of max-plus matrices.
These applications include
communication networks~\cite{BH-00},
cyclic scheduling~\cite{CMQV-89},
link-reversal algorithms~\cite{CFWW-15},
network synchronizers~\cite{ER-97},
as well as transportation and manufacturing systems~\cite{BCOQ-92}.

An analog of the ultimate pseudoperiodicity of irreducible max-plus matrices
for real-valued non-negative matrices is provided by the Perron-Frobenius
theorem.
The largest mean cycle weight~$\lambda$ is analogous to the spectral radius of
non-negative matrices.
The spectral radius of stochastic matrices, i.e., non-negative matrices whose row sums are all~$1$, is equal to~$1$.
The sequence of powers of an irreducible stochastic matrix is known to converge
to a rank-$1$ stochastic matrix and the rate of convergence is known to
depend on the matrix's spectral gap, i.e., the difference between its largest
and its second largest eigenvalue.
Similarly, the transient of an irreducible max-plus matrix depends on the difference between the largest and the second-largest mean cycle weight.

If all matrix entries in $A$ are restricted to $\bzero=-\infty$ or $\bunity=0$ then 
the tropical matrix algebra becomes Boolean matrix algebra (i.e., linear algebra over the Boolean semiring),
and the associated digraph becomes unweighted.
Powers of Boolean matrices have been thoroughly studied in combinatorics (see, e.g.,~\cite{BR}),
and various bounds on their transient of periodicity, called index or exponent in these cases, have been obtained.
One well-known application is the Frobenius coin problem, which can be seen as calculating the transient of a specifically constructed graph.
For general connected graphs, Wielandt~\cite{Wie-50} proved the bound~$(n-1)^2+1$ and Dulmage and Mendelsohn~\cite{DM-64} proved $g(n-2)+n$,
where~$g$ is the girth of the graph.


Note that the same problem can be considered locally: for each pair $i,j$,
given $\gamma$ and $\lambda$ that work for \eqref{e:period}, find
the minimal $T_{i,j}$ such that 
\begin{equation}
\label{e:periodij} 
\forall t\geq T_{i,j}\colon\quad
(A^{t+\gamma})_{i,j}=\lambda^{\otimes \gamma}\otimes (A^t)_{i,j}.
\end{equation}
Denoting that minimal $T_{i,j}$ by $T_{i,j}(A)$, we can also consider 
$\max_{i\in [n]} T_{i,j}(A)$ and $\max_{j\in [n]} T_{i,j}(A)$. These quantities are called the {\em transient of the $j$th column} of $A$ and the {\em transient of the $i$th row} of $A$, respectively. The bounds for such transients can 
be much lower than those for $T(A)$, as it was shown by~\cite{MNSS} in an important special case of critical 
rows and columns, i.e., in the case where the 
index of the row or column corresponds to a critical node (see Definition
\ref{def:lambdacrit} below). 

The eventual periodicity was reformulated and generalized
by Sergeev and Schneider~\cite{Ser-09,SS-11} via the concept of CSR expansions. They observed that in the case of 
eventual periodicity of $\{A^t\}_{t\geq 1}$ there is a big enough $T$ for which we have
\begin{equation}
\label{e:CSRperiod}
\forall t\geq T\colon \quad 
A^{t}=CS^tR,
\end{equation}
where $C,S$ and~$R$ are constructed from $A$ (see Definition~\ref{def:CSR} below) in such a way that $CS^tR$
is a purely pseudoperiodic sequence (i.e. $CS^{t+\gamma}R=\lambda^{\otimes\gamma}\otimes CS^{t}R$ for any~$t$
and some $\lambda$ and $\gamma$). The smallest $T$ for which~\eqref{e:CSRperiod} holds is equal to $T(A)$.

In~\cite{MNS}, we used this approach to unify and improve the known bounds
on~$T(A)$.
To this aim, we introduced the so-called 
{\it weak CSR expansions}. Observe that, given any $A\in\Rmax^{n\times n}$, there exists a  
big enough $T$ for which 
\begin{equation}
\label{e:weakcsr}
\forall t\geq T\colon\quad A^t=CS^tR\oplus B^t. 
\end{equation}
Here $C$, $S$, and $R$ are defined as in~\eqref{e:CSRperiod} and $B$ is a matrix obtained from $A$
by setting some entries to $\bzero$. The smallest number $T$ for which~\eqref{e:weakcsr} holds is 
called the weak CSR expansion threshold and denoted by $T_1(A,B)$.
In this paper we consider only the case $B=\bnacht$, 
where all entries with an index corresponding to a critical node are set to~$\0$ (see Definition~\ref{def:CSR} below). In this case $T_1(A,B_N)$ will be abbreviated to $T_1(A)$.

For irreducible \emph{Boolean} matrices, $T(A)$ and~$T_1(A,B_N)$ coincide. For reducible ones, they coincide, as soon as $T_1(A,B_N)\ge n$.
This allows to hope for extensions of bounds on exponents of graph to $T_1$ in the framework of max-plus algebra,
In particular, Wielandt and Dulmage-Mendelsohn bounds were extended in Theorem~4.1 of~\cite{MNS}.

However, no information was given in~\cite{MNS} on the question of which classes of matrices attain these bounds for $T_1(A,B)$. 
For the index of digraphs, those results are well-known. In particular, the digraph attaining the Wielandt bound is unique
up to renumbering the nodes, and the digraphs attaining the bound of Dulmage-Mendelsohn were studied  by Shao~\cite{Shao}.
The main results of the present paper are Theorem~\ref{t:mainres}, which characterizes all matrices~$A$ (or all weighted digraphs)
such that $T_1(A,\bnacht)$ attains the Dulmage-Mendelsohn bound, and Theorem~\ref{t:wiel}, which characterizes 
those attaining the Wielandt bound. Unfortunately, we have not been able to characterize the matrices that reach 
the bounds for other
choices of~$B$ studied in~\cite{MNS}.

On the other hand, in~\cite{MNSS} we had proved that the same bounds (and
others) also apply to the transient of a critical row or column of matrices.
Theorems~\ref{t:CritColDM} and~\ref{t:CritColWiel} characterize the matrices
for which these bounds are reached.

The paper is organized as follows.
After giving preliminary definitions and results in Section~\ref{sec:prelim},
we state our characterizations in Section~\ref{sec:thms} and give a quick overview of how to prove them.
The characterizations for the Dulmage-Mendelsohn bound on the weak CSR
threshold is then proved in
Section~\ref{sec:dm} and that for the Wielandt bound in Section~\ref{sec:wiel}.
Finally, in Section~\ref{sec:crit}, we prove the characterizations for the
transients of critical rows and columns.

\section{Preliminaries}\label{sec:prelim}

\subsection{Digraphs and Walks}
Most of the techniques for analyzing the max-plus matrix powers and their behavior are 
based on consideration of  walks on the
associated weighted digraphs. Hence it is essential 
to introduce the notion of weighted digraph associated with a given
max-plus matrix, as well as the related notions of walks, connectivity, girth and cyclicity.

\begin{definition}[Associated digraph]
\label{def:digr}
Let $A\in\Rmax^{n\times n}$. The {\em digraph associated with $A$}, denoted by 
$\digr(A)$, is defined
as the pair $(N,E)$ where $N=\{1,\ldots,n\}$ is the set of nodes
and $E=\{(i,j)\in N\times N\colon a_{i,j}\neq\bzero\}$ 
is the set of arcs connecting these nodes. Arc $(i,j)$ has {\em weight}
$a_{i,j}$.
\end{definition}

\begin{definition}[Walks]
\label{def:walks}
A sequence $i_0\dots i_k$, where $i_0,\ldots,i_k\in N$, is called a {\em walk} on a digraph 
$\digr=(N,E)$ if for any $s\colon 1\leq s\leq k$ the arc 
$(i_{s-1},i_{s})$ is in $E$.\\ 
For a walk $W=i_0\dots i_k$ we define the {\em length}
of the walk as $l(W):=k$, which is the number of letters in that walk
(as a sequence of letters) minus one.\\
If $i_0=i_k$ then $W=i_0\dots i_k$ is called
 {\em closed}.\\
A closed walk with no proper closed subwalk is called a {\em cycle}.\\
If $\digr=\digr(A)$ and $W=i_0\dots i_k$ then the {\em weight} of $W$
is defined as $p(W):=a_{i_0i_1}\otimes\ldots\otimes a_{i_{k-1}i_k}$.
\end{definition}

\begin{definition}[Walk sets]
\label{def:walksets}
Let $\walks$ be a set of walks on a weighted  digraph $\digr$, and 
let $\subcrit$ be a subdigraph of $\digr$. Denote
$p(\walks):=\max\{p(W)\colon W\in\walks\}$
(the maximal weight of all walks in $\walks$).

The following sets of walks will be particularly useful.
\begin{enumerate}
\item $\walkslen{i}{j}{t}:$ the set of walks from $i$ to $j$ that have length $t$.
\item $\walkslen{i}{j}{t,\gamma}:$ the set of walks from 
$i$ to $j$ that have length $t$ modulo $\gamma$.
\item $\walkslennode{i}{j}{t,\gamma}{\subcrit}:$ 
the set of walks of 
$\walkslen{i}{j}{t,\gamma}$
that go through a node in $\subcrit$.
\end{enumerate} 

\end{definition}

Observe that a sequence of nodes is not necessarily a walk. Also observe that an easy way to change 
a walk into another walk is to remove closed subwalks from a given walk,
or to replace a (possibly empty) subwalk by another walk with the same start and same end. This will be the main tool of this article.

The following optimal walk interpretation of matrix powers is well-known:
\begin{equation}\label{e:At}
(A^t)_{i,j}=p(\walkslen{i}{j}{t}).
\end{equation}
See, for example, \cite[Example 1.2.3]{But}. 

We now give some definitions related to connectivity in digraphs.

\begin{definition}[Connectivity]
 A digraph $\digr=(N,E)$ is called {\em strongly connected} if for each $i,j\in N$ there is a walk from
$i$ to $j$.\\
A digraph $\digr'=(N',E')$ is called a {\em sub(di)graph} of $\digr$ if $N'\subseteq N$ and 
\mbox{$E'\subseteq E\cap (N'\times N').$}\\ 
{\em Maximal strongly connected component} of $\digr$, further abbreviated to {\em s.c.c.}, is a maximal strongly 
connected subgraph of $\digr$.\\ 
A digraph is called {\em completely reducible} if 
for any pair of s.c.c. of $\digr$ there is no walk connecting a node of one s.c.c. to a node of another 
s.c.c.
\end{definition}

\begin{definition}[Maximal Girth]
\label{def:girth}
For a strongly connected digraph $\digr$, the {\em girth} of $\digr$ is defined as 
the minimal length of a cycle in $\digr$.\\ 
For a completely reducible digraph $\digr$, the {\em maximal girth} of $\digr$, denoted by $\g(\digr)$,
is defined as the maximal girth of the s.c.c.'s of $\digr$.
\end{definition}
Although we use the same notation~$\g(\digr)$, note that this is not what is usually called the girth of a reducible graph,
namely the least common multiple of the girths of its s.c.c.'s, a quantity not used in this paper.

\begin{definition}[Cyclicity]
\label{def:cyclicity}
For a strongly connected graph $\digr$, the {\em cyclicity} of $\digr$, denoted by $\gamma(\digr)$, is defined as the greatest common divisor
of all cycle lengths of $\digr$.\\
For a completely reducible digraph of $\digr$, the 
{\em cyclicity} of $\digr$ is defined as the least common multiple of the cyclicity of
the strongly connected components of $\digr$.
\end{definition}

In max-plus algebra we deal not only with $\digr(A)$ but also with special subgraphs 
of it such as the critical graph of the following definition.

\begin{definition}[Maximum cycle mean and critical graph]
\label{def:lambdacrit}
The {\em maximum cycle mean} of $A$ is
\begin{equation}
\label{e:mcm}
\lambda(A)=\max_{i_1,\ldots,i_k} (a_{i_1i_2}\otimes\ldots\otimes 
a_{i_{k-1}i_k}\otimes a_{i_ki_1})^{\otimes 1/k}.
\end{equation}
The {\em critical graph} of $A$, denoted by $\crit(A)$, 
is a subdigraph of $\digr(A)$ consisting of all nodes and arcs of the
cycles $i_1\dots i_ki_1$ that attain the maximum in~\eqref{e:mcm}.
Such nodes and arcs are also called {\em critical}.
\end{definition}

\begin{definition}[Visualization]
We say that~$A$ is {\em visualized } if $a_{i,j} \leq \lambda(A)$ for
all~$i$ and~$j$ and $a_{i,j}=\lambda(A)$ whenever $(i,j)$ is an arc of $\crit(A)$.
It is {\em strictly visualized } if it is visualized and $a_{i,j}=\lambda(A)$ if and
only if $(i,j)$ is an arc of $\crit(A)$.

A {\em scaling } of~$A$ is 
a matrix of the form $B = D^{-}AD$ where~$D$ is a
diagonal matrix with finite diagonal entries.
A {\em visualization } of~$A$ is a scaling that is visualized.
Likewise, a {\em strict visualization } of~$A$ is a scaling that is strictly
visualized.
\end{definition}

\begin{theorem}[\cite{SS-91}]
Every~$A$ with $\lambda(A)\neq\0$ has a strict visualization.
\end{theorem}

\subsection{Weak CSR Expansion} 

We now present important definitions, notations and facts related to the main theme of this work.

\begin{definition}[Kleene star]
Let $A\in\Rpnn$ with 
$\lambda(A)\leq \1$. Then
\begin{equation*}
A^*:=I\oplus A\oplus\ldots\oplus A^{n-1}
\end{equation*}
is called the {\em Kleene star} of $A$. Recall that $I$ denotes the max-plus identity matrix
(which has $\bunity$ on the diagonal and $\bzero$ off the diagonal).
\end{definition}

\begin{definition}[CSR]
\label{def:CSR}
Let $A\in\Rpnn$.
If $\lambda(A)\neq\0$, set
$M=((\lambda(A)^-\otimes A)^{\gamma})^*$, where 
$\gamma$ is the 
cyclicity of $\crit(A)$, and define matrices $C,S$ and $R$ by
\begin{equation*}
\begin{split}
c_{i,j}&:=
\begin{cases}
m_{i,j},& \text{for $j\in \crit(A)$},\\
\bzero,& \text{otherwise}, 
\end{cases},\quad
r_{i,j}:=
\begin{cases}
m_{i,j},& \text{for $i\in \crit(A)$},\\
\bzero,& \text{otherwise}, 
\end{cases},\\
s_{i,j}&:=
\begin{cases}
a_{i,j}, &\text{for $(i,j)\in\crit(A)$},\\
\bzero, &\text{otherwise.}
\end{cases}
\end{split}
\end{equation*}
If $\lambda(A)=\0$, let~$CS^tR$ be the matrix in~$\Rpnn$ with only~$\0$ entries for any~$t$.
\end{definition}

This definition is best understood in combination with Proposition~\ref{p:CSRprops} part~\ref{CSR-walks}, which 
gives an optimal walk interpretation of $(CS^tR[A])_{i,j}$: the maximal weight of walks connecting $i$ to 
$j$ that have length modulo $\gamma$. Optimal walk interpretation also gives an idea why we have ``division'' by $\lambda(A)$ in the definition of $M$ and hence $C$ and $R$: the lengths of corresponding optimal walks are not
controlled.  
Informally speaking, CSR is related to turnpike theorems in a discrete deterministic case~\cite{AGW-05} at least in some optimal long walks most of the material should be 
concentrated on the critical graph so that $A^t$ is determined by $S^t$ for large enough $t$. This also 
shows why there is no ``division'' by $\lambda(A)$ in $S$.

Below we also deal with some auxiliary matrices, for which the
CSR terms are (a priori) different from those derived from $A$. Therefore we will write $CS^tR[A]$ for a CSR term 
derived from $A$.

\begin{definition}[$\bnacht$]
\label{def:BN}
The Nachtigall matrix $B = \bnacht$ is defined as the matrix whose entries are
\begin{equation*}
(\bnacht)_{i,j}=
\begin{cases}
\bzero, & \text{if $i$ or $j$ is a critical node}\\
a_{i,j}, & \text{else.}
\end{cases}
\end{equation*}
\end{definition}

This is the most obvious choice of a matrix $B$ that appears in a weak CSR expansion, since if $B=\bnacht$ then
$B^t_{i,j}$ expresses the optimal weight of all walks that connect $i$ to $j$ and do not 
touch any critical node of $A$.

\begin{definition}[$T_1(A)=T_1(A,\bnacht)$]
The {\em weak CSR threshold } $T_1(A,\bnacht)$ is the least $T$, for which
\begin{equation*}
A^t=CS^tR\oplus \bnacht^t,\quad  t\geq T
\end{equation*}
holds.
In the sequel, $T_1(A,\bnacht)$ is abbreviated to $T_1(A)$.
\end{definition}

We will further work with the following two bounds on $T_1(A)$, which originate in the
works on digraph exponents or indices of imprimitivity and (in the case of unweighted digraphs 
and matrix powers over Boolean algebra) are due to Wielandt~\cite{Wie-50} and Dulmage and Mendelsohn~\cite{DM-64}.

\begin{definition}[$\wiel(n)$ and $\DM(g,n)$]
\label{def:WiDM}
For any $n\in\N$ (the set of natural numbers)
and any $1\le g\le n$, we define
\begin{equation}
\begin{split}
\wiel(n)&=
\begin{cases}
0, & \text{if $n=1$},\\
(n-1)^2+1, & \text{otherwise}.
\end{cases}\\
\DM(g,n)&= g(n-2)+n.
\end{split}
\end{equation}
\end{definition}

\begin{theorem}[\cite{MNS} Theorem 4.1]
\label{t:bounds}
For $A\in\Rpnn$ and $g=\g(\crit(A))$, we have:
$T_1(A)\leq\min(\wiel(n),\DM(g,n))$.
\end{theorem}

\begin{proposition}[\cite{MNS},\cite{SS-11}]
\label{p:CSRprops}
Let $A\in\Rmax^{n\times n}$ have $\lambda(A)=\1$.
\begin{enumerate}[(i)]
\item\label{CSR-walks} CSR terms have the following 
{\bf optimal walk interpretation}:
\begin{equation*}
(CS^tR[A])_{i,j}=p(\walkslennode{i}{j}{t,\gamma}{\crit(A)})\quad\forall i,j\in\{1,\ldots,n\}
\end{equation*}
for $\gamma$ being any multiple of $\gamma(\crit(A))$.
\item\label{CSR-period} $CS^{t+\gamma}R[A]=CS^tR[A]$ for all $t\geq 1$ {\bf (periodicity)}.
\item\label{CSR-group}  
$CS^{t_1}R[A]CS^{t_2}R[A]=CS^{t_1+t_2}R[A]$ for all $t_1,t_2\geq 1$ {\bf (group law)}.
\item\label{CSR-limit}
$\lim_{k\to\infty} A^{t+k\gamma}=CS^tR[A]\quad\forall t>0$. {\bf (limit property)}
\item\label{CSR-pseudoGpe} $A(CS^tR[A])=(CS^tR[A])A=CS^{t+1}R[A]$
\end{enumerate}
Parts~(\ref{CSR-group}) and~(\ref{CSR-pseudoGpe}) also hold with general $\lambda(A)$.
\end{proposition}
\begin{proof}
\eqref{CSR-walks}: This property follows from~\cite{SS-11} Theorem 3.3, or
\cite{MNS} Theorem 6.1 where a more general statement is given.\\
\eqref{CSR-period}, \eqref{CSR-group}: These properties are shown in~\cite{SS-11} Proposition 3.2 and
Theorem 3.4.\\
\eqref{CSR-limit}: It is obvious that $\lambda(B)<\bunity$ and therefore $\lim B^t=\bzero$.
The claim then follows from the weak CSR expansion $A^t=CS^tR\oplus B^t$ and the periodicity
of $\{CS^tR\}_{t\geq 1}$ (ii).\\
\eqref{CSR-pseudoGpe} can be deduced from~\eqref{CSR-limit} (as could \eqref{CSR-group}).

Extension to general $\lambda(A)$ follows from the homogeneity of (iii) and (v).
\end{proof}

\section{Theorem Statements and Proof Strategy}
\subsection{Statements}\label{sec:thms}
For any matrix $A \in\Rmax^{n\times n}$ and any $1\leq g\leq n$, we define
the following matrices.

\begin{equation}
\label{e:A1}
(A_1)_{i,j}= 
\begin{cases}
a_{i,j} & \text{if $j=i+1$ and $1\leq i\leq n-1$ }\\
&\text{or $(i,j)\in\left\{(n,1),(g,1)\right\}$},\\
\0 & \text{otherwise}.
\end{cases}
\end{equation}

\begin{equation}
\label{e:B1}
(B_1)_{i,j}=
\begin{cases}
a_{i,j}, &\text{if $i>g,j>g$ and $j\equiv_g i+1$ }\\
\0, &\text{otherwise}
\end{cases}
\end{equation}

\begin{equation}
\label{e:A2}
(A_2)_{i,j}=
\begin{cases}
\0, &\text{if $(A_1\oplus B_1)_{i,j}>\0$},\\
a_{i,j}, &\text{otherwise}.
\end{cases}
\end{equation}

\begin{figure}
\centering
\begin{tikzpicture}[>=latex']
\colorlet{darkgreen}{black!50!green}
\node[circle,very thick,draw,minimum size=9mm] (n1) at (0,0) {$1$};
\node[circle,very thick,draw,minimum size=9mm] (n2) at (-1.5,1.7) {$2$};
\node[circle,very thick,draw,minimum size=9mm] (n3) at (0,3.4) {$\scriptstyle g=3$};
\node[circle,very thick,draw,minimum size=9mm] (n4) at (2,3.4) {$4$};
\node[circle,very thick,draw,minimum size=9mm] (n5) at (4,3.4) {$5$};
\node[circle,very thick,draw,minimum size=9mm] (n6) at (6,3.4) {$6$};
\node[circle,very thick,draw,minimum size=9mm] (n7) at (8,3.4) {$7$};
\node[circle,very thick,draw,minimum size=9mm] (n8) at (9.5,1.7) {$8$};
\node[circle,very thick,draw,minimum size=9mm] (n9) at (8,0) {$9$};
\node[circle,very thick,draw,minimum size=9mm] (n10) at (6,0) {$10$};
\node[circle,very thick,draw,minimum size=9mm] (n11) at (4,0) {$11$};
\node[circle,very thick,draw,minimum size=9mm] (n12) at (2,0) {$12$};

\draw[-,red] (n4) -- (n5);
\draw[-,red] (n5) -- (n6);
\draw[-,red] (n6) -- (n7);
\draw[-,red] (n7) -- (n8);
\draw[-,red] (n8) -- (n9);
\draw[-,red] (n9) -- (n10);
\draw[-,red] (n10) -- (n11);
\draw[-,red] (n11) -- (n12);
\draw[->,red] (n5) -- (n9);
\draw[->,red] (n5) -- (n12);
\draw[->,red] (n10) -- (n8);
\draw[->,red] (n10) -- (n5);

\draw[->,very thick,,loosely dotted,darkgreen] (n1) -- (n2);
\draw[->,very thick,,loosely dotted,darkgreen] (n2) -- (n3);
\draw[->,very thick,,loosely dotted,darkgreen] (n3) -- (n4);

\draw[->,very thick,,loosely dotted,darkgreen] (n4) -- (n5);
\draw[->,very thick,,loosely dotted,darkgreen] (n5) -- (n6);
\draw[->,very thick,,loosely dotted,darkgreen] (n6) -- (n7);
\draw[->,very thick,,loosely dotted,darkgreen] (n7) -- (n8);
\draw[->,very thick,,loosely dotted,darkgreen] (n8) -- (n9);
\draw[->,very thick,,loosely dotted,darkgreen] (n9) -- (n10);
\draw[->,very thick,,loosely dotted,darkgreen] (n10) -- (n11);
\draw[->,very thick,,loosely dotted,darkgreen] (n11) -- (n12);
\draw[->,very thick,,loosely dotted,darkgreen] (n12) -- (n1);
\draw[->,very thick,,loosely dotted,darkgreen] (n3) -- (n1);

\end{tikzpicture}
\caption{Example of the digraph of~$A_1$ (dotted arcs) and $B_1$ (solid arcs)
For $B_1$, only some of the arcs are shown.
}
\label{fig:graph:decomp}
\end{figure}

Figure~\ref{fig:graph:decomp} shows an example of~$A_1$ and~$B_1$.
The definitions of $A_1$, $B_1$ and $A_2$ imply that
\begin{equation}
\label{e:repr}
A=A_1\oplus B_1\oplus A_2.
\end{equation}

\begin{theorem}\label{t:mainres}
Let $A\in\Rpnn$ with $g= \g(\crit(A))\geq 2$. 
Then $A$ satisfies $T_1(A)=\DM(g,n)$ if and only if there exists a renumbering of nodes such that
the following conditions hold.

\begin{enumerate}[1.]
\item\label{Cgnprime} $g$ and $n$ are coprime;

\item\label{CGc} $\crit(A)$ is strongly connected with a unique  critical cycle of length~$g$ up to choice of its first node;

\item\label{CZO} $1\cdots g1$ is critical

\item\label{CA2} $A_2<CSR[A_1]$;

\item\label{CB1} $\lambda(A)^{\otimes(j-i-1)}\otimes (B_1)_{i,j}<(A_1)_{i,j}^{j-i} \text{ when } j>i+1\;\text{, } i>g \text{ and } j\equiv_g (i+1)$;

\item\label{CB2} $(B_1)_{g+1,n}^{\DM(g,n)-1}<(CS^{\DM(g,n)-1}R)_{g+1,n}[A_1]$,
\end{enumerate}
where $A_1$, $B_1$ and $A_2$ are defined as in~\eqref{e:A1},~\eqref{e:B1} and~\eqref{e:A2}.

The renumbering satisfying Conditions 1--6 is necessarily unique.
More precisely, it is the only one that ensures that
\begin{itemize}
 \item $1\cdots n1$ is an Hamiltonian cycle of $\digr(A)$ with the largest weight, which is unique up to choice of its first node.
\item $1\cdots g1$ is critical.
\end{itemize}
\end{theorem}

\begin{remark}
\label{r:irred}
Note that in the above theorem we do not assume that $A$ is irreducible.
The same is true about all the statements in this section.
However, it follows from Condition~\ref{CA2} that $A_1$ is irreducible (and aperiodic) and thus, so is~$A$.
\end{remark}

\begin{remark}
\label{r:n<2g}
The case $g=1$ turns out to be much more complicated. Although some results do apply (e.g. Proposition~\ref{p:critgraph})
we were not able to characterize the matrices reaching the bound. Notice that already in the Boolean case the situation is more complicated
and not completely understood (see~\cite{Shao}).

On the other hand, if $n<2g$ the situation is simpler : $j\equiv_g i+1$ with $i,j>g$ holds if and only if $j=i+1$ and $i,j>g$. In 
this case Conditions~\ref{CB1} and~\ref{CB2} above hold automatically. 
For Condition~\ref{CB2}, note that $\digr(B_1)$ is acyclic hence $B_1^{n-g}=\0$.   
\end{remark}

\begin{remark}
The index of $\digr(A_1)$ reaches the bound~$\DM(g,n)$.
It is easy to recover the characterization of such graphs obtained in~\cite{Shao} from the theorem.
\end{remark}

Let us see what Theorem~\ref{t:mainres} means on an example.
\begin{example}
We fix~$g=3$ and~$n=8$.
The theorem says that any matrix of this type that reaches the bound can be decomposed
as in Equations~\eqref{e:A1}, \eqref{e:B1},\eqref{e:A2} with $1231$ as critical cycle.
Let us assume that $(A_1)_{1,2}=(A_1)_{2,3}=(A_1)_{3,1}=0$ and all other finite entries of $A_1$ equal~$-1$.
(See. Figure~\ref{fig:graph38}.)

\begin{figure}
\centering
\begin{tikzpicture}[>=latex']
\node[circle,very thick,draw,minimum size=9mm] (n1) at (0,0) {$1$};
\node[circle,very thick,draw,minimum size=9mm] (n2) at (-2,2) {$2$};
\node[circle,very thick,draw,minimum size=9mm] (n3) at (0,4) {$\scriptstyle g=3$};
\node[circle,very thick,draw,minimum size=9mm] (n4) at (3,4) {$4$};
\node[circle,very thick,draw,minimum size=9mm] (n5) at (6,4) {$5$};
\node[circle,very thick,draw,minimum size=9mm] (n6) at (8,2) {$6$};
\node[circle,very thick,draw,minimum size=9mm] (n7) at (6,0) {$7$};
\node[circle,very thick,draw,minimum size=9mm] (n8) at (3,0) {$8$};

\draw[-,red] (n4) -- (n5);
\draw[-,red] (n5) -- (n6);
\draw[-,red] (n6) -- (n7);
\draw[-,red] (n7) -- (n8);
\draw[->,red] (n4) to node [left] {$\scriptstyle <-4$} (n8);
\draw[->,red] (n6) to node [below left] {$\scriptstyle <2$} (n4);%
\draw[->,red] (n7) to node [right] {$\scriptstyle <2$} (n5);%
\draw[->,red] (n8) to node [above left] {$\scriptstyle <2$} (n6);

\draw[->,very thick,,loosely dotted,darkgreen] (n1) to node [below left] {$0$} (n2);%
\draw[->,very thick,,loosely dotted,darkgreen] (n2) to node [above left] {$0$} (n3);%
\draw[->,very thick,,loosely dotted,darkgreen] (n3) to node [left] {$0$} (n1);%

\draw[->,very thick,,loosely dotted,darkgreen] (n3) to node [above ] {$-1$} (n4);
\draw[->,very thick,,loosely dotted,darkgreen] (n4) to node [above ] {$-1$} (n5);
\draw[->,very thick,,loosely dotted,darkgreen] (n5) to node [above right] {$-1$} (n6);
\draw[->,very thick,,loosely dotted,darkgreen] (n6) to node [below right] {$-1$} (n7);
\draw[->,very thick,,loosely dotted,darkgreen] (n7) to node [below] {$-1$} (n8);
\draw[->,very thick,,loosely dotted,darkgreen] (n8) to node [below] {$-1$} (n1);

\end{tikzpicture}
\caption{Digraph of the example with $g=3$ and~$n=8$}
\label{fig:graph38}
\end{figure}

Let us compute $CSR[A_1]$ with Scicoslab (a fork from Scilab which incorporates the MaxPlus Toolbox).
We get
$$
CSR[A_1]=
\begin{pmatrix}
 - 12 &    \color{darkgreen}{ 0} &   - 6 &   - 13 &  - 2 &   - 9 &   - 16 &  - 5 \\ 
  - 6 &   - 12 &    \color{darkgreen}{0} &   - 7 &   - 14 &  - 3 &   - 10 &  - 17\\  
   \color{darkgreen}{ 0} &   - 6 &   - 12 &  \color{darkgreen}{- 1} &   - 8 &   - 15 &  - 4 &   - 11\\  
  - 11 &  - 17 &  - 5 &   - 12 &  \color{darkgreen}{- 19} &  - 8 &   - 15 &  \color{red}{- 22}\\  
  - 4 &   - 10 &  - 16 &  - 5 &   - 12 &  \color{darkgreen}{- 19} &  - 8 &   - 15\\  
  - 15 &  - 3 &   - 9 &   \color{red}{- 16 }&  - 5 &   - 12 &  \color{darkgreen}{- 19} &  - 8\\   
  - 8 &   - 14 &  - 2 &   - 9 &   \color{red}{- 16} &  - 5 &   - 12 & \color{darkgreen}{ - 19}\\  
  \color{darkgreen}{- 1} &   - 7 &   - 13 &  - 2 &   - 9 &   \color{red}{- 16} &  - 5 &   - 12\\ 
\end{pmatrix}
$$
The colored entries are those for which $(A_2)_{i,j}=\0$ and thus Condition~\ref{CA2} is trivial.

For those entries, Conditions~\ref{CGc} and~\ref{CZO} mean that $(B_1)_{6,4}, (B_1)_{7,5}, (B_1)_{8,6}<2$,
while Condition~\ref{CB1} means that $(B_1)_{4,8}<-4$.

Finally, get

$$T_1(A)=\DM(g,n) \iff 
\left\{\begin{array}{c}
        (B_1)_{4,8}^{25}<-22 \textnormal{ (Condition~\ref{CB2}) }\\ 
 \\
 A<\begin{pmatrix}
 - 12 &    \color{darkgreen}{ 0^=} &   - 6 &   - 13 &  - 2 &   - 9 &   - 16 &  - 5 \\ 
  - 6 &   - 12 &    \color{darkgreen}{0^=} &   - 7 &   - 14 &  - 3 &   - 10 &  - 17\\  
   \color{darkgreen}{ 0^=} &   - 6 &   - 12 &  \color{darkgreen}{- 1^=} &   - 8 &   - 15 &  - 4 &   - 11\\  
  - 11 &  - 17 &  - 5 &   - 12 &  \color{darkgreen}{- 1^=} &  - 8 &   - 15 &  \color{red}{- 4}\\  
  - 4 &   - 10 &  - 16 &  - 5 &   - 12 &  \color{darkgreen}{- 1^=} &  - 8 &   - 15\\  
  - 15 &  - 3 &   - 9 &   \color{red}{2}&  - 5 &   - 12 &  \color{darkgreen}{- 1^=} &  - 8\\   
  - 8 &   - 14 &  - 2 &   - 9 &   \color{red}{2} &  - 5 &   - 12 & \color{darkgreen}{ - 1^=}\\  
  \color{darkgreen}{- 1^=} &   - 7 &   - 13 &  - 2 &   - 9 &   \color{red}{2} &  - 5 &   - 12\\ 
\end{pmatrix}

      \end{array}\right. ,$$
      
where the matrix inequality means that each entry of~$A$ should be equal to the corresponding entry if it is denoted by $^=$ and strictly less than it otherwise.

A computation of~$B_1^{25}$ shows that Condition~\ref{CB2} can not be removed.
For instance, it will be satisfied for~$(B_1)_{6,4}=(B_1)_{7,5}=(B_1)_{8,6}=-1$ not for~$(B_1)_{6,4}=(B_1)_{7,5}=(B_1)_{8,6}=0$.

\end{example}

In~\cite{MNSS}[Lemma~8.2], we have noticed that for any matrix $A$ with $\lambda(A)\neq\0$, the maximal transient of its critical rows or columns is at least the index of its critical graph,
so that if the index of $\crit(A)$ reaches the bound, so does the transient of one row and one column.
The following shows that the converse is also true.

\begin{theorem}\label{t:CritColDM}
Let $A\in\Rpnn$ and $g= \g(\crit(A))$.  
Then the transient of the critical rows and columns of~$A$ is equal
to~$\DM(g,n)$
if and only if
its critical graph has index~$\DM(g,n)$.
\end{theorem}

\begin{theorem}
\label{t:wiel}
Let $A\in\Rpnn$.
Then, $T_1(A)=\wiel(n)$ if  and only if
there exists a renumbering of nodes such that
\begin{itemize}
\item[1.] $\g(\crit(A))=n-1$ and $1\cdots(n-1)1$ is critical, or $\g(\crit(A))=n$
and $1\cdots n1$ is critical,
\item[2.] $A_2<CSR[A_1]$,
\end{itemize}
where $A_1$ and $A_2$ are defined as in~\eqref{e:A1} and~\eqref{e:A2} with
$g=n-1$ in both cases.

The renumbering satisfying these conditions is necessarily unique.
More precisely, it is the only one that ensures that 

\begin{itemize}
 \item $1\cdots n1$ is an Hamiltonian cycle with the largest weight, which is unique up to choice of its first node.
 
 \item $1\cdots(n-1)1$ is a cycle of length~$n-1$ with the largest weight, which is unique up to choice of its first node.
\end{itemize}

\end{theorem}

\begin{remark}
The digraph $\digr(A_1)$ is exactly the unique (up to renumbering) digraph whose index reaches the bound~$\wiel(n)$.
\end{remark}

Let us see what Theorem~\ref{t:wiel} means on an example.
\begin{example}
We fix~$g=n=8$.
The theorem says that any matrix of this type that reaches the bound can be decomposed
as in Equations~\eqref{e:A1} and ,\eqref{e:A2} with $g=7$.
Let us assume that $(A_1)_{i,i+1}=(A_1)_{8,1}=0$ and $(A_1)_{7,1}=-1$.
(See. Figure~\ref{fig:graphWiel8}.)

\begin{figure}
\centering
\begin{tikzpicture}[>=latex']
\node[circle,very thick,draw,minimum size=9mm] (n1) at (0,0) {$7$};
\node[circle,very thick,draw,minimum size=9mm] (n2) at (-2,1) {$8$};
\node[circle,very thick,draw,minimum size=9mm] (n3) at (0,2) {$1$};
\node[circle,very thick,draw,minimum size=9mm] (n4) at (2,2) {$2$};
\node[circle,very thick,draw,minimum size=9mm] (n5) at (4,2) {$3$};
\node[circle,very thick,draw,minimum size=9mm] (n6) at (6,1) {$4$};
\node[circle,very thick,draw,minimum size=9mm] (n7) at (4,0) {$5$};
\node[circle,very thick,draw,minimum size=9mm] (n8) at (2,0) {$6$};

\draw[->,very thick,,loosely dotted,darkgreen] (n1) to node [below left] {$0$} (n2);%
\draw[->,very thick,,loosely dotted,darkgreen] (n2) to node [above left] {$0$} (n3);%
\draw[->,very thick,,loosely dotted,darkgreen] (n3) to node [above ] {$0$} (n4);
\draw[->,very thick,,loosely dotted,darkgreen] (n4) to node [above ] {$0$} (n5);
\draw[->,very thick,,loosely dotted,darkgreen] (n5) to node [above right] {$0$} (n6);
\draw[->,very thick,,loosely dotted,darkgreen] (n6) to node [below right] {$0$} (n7);
\draw[->,very thick,,loosely dotted,darkgreen] (n7) to node [below] {$0$} (n8);
\draw[->,very thick,,loosely dotted,darkgreen] (n8) to node [below] {$0$} (n1);
\draw[->,very thick,,loosely dotted,darkgreen] (n1) to node [right] {$-1$} (n3);%

\end{tikzpicture}
\caption{Digraph of the example with $g=n=8$}
\label{fig:graphWiel8}
\end{figure}

Let us compute $CSR[A_1]$ with Scicoslab.
We get
$$
CSR[A_1]=
\begin{pmatrix}

 - 7&  \color{darkgreen}{  0}&  - 1&  - 2&  - 3&  - 4&  - 5&  - 6\\  
  - 6&  - 7&    \color{darkgreen}{  0}&  - 1&  - 2&  - 3&  - 4&  - 5\\  
  - 5&  - 6&  - 7&    \color{darkgreen}{  0}&  - 1&  - 2&  - 3&  - 4\\  
  - 4&  - 5&  - 6&  - 7&    \color{darkgreen}{  0}&  - 1&  - 2&  - 3\\  
  - 3&  - 4&  - 5&  - 6&  - 7&    \color{darkgreen}{  0}&  - 1&  - 2\\  
  - 2&  - 3&  - 4&  - 5&  - 6&  - 7&    \color{darkgreen}{  0}&  - 1\\  
  \color{darkgreen}{ - 1}&  - 2&  - 3&  - 4&  - 5&  - 6&  - 7&    \color{darkgreen}{  0}\\  
    \color{darkgreen}{  0}&  - 1&  - 2&  - 3&  - 4&  - 5&  - 6&  - 7
\end{pmatrix}
$$
The colored entries are those for which $A_{i,j}=(A_1)_{i,j}$ and thus there is nothing to check.
$A$ satisfies $T_1(A)=\Wi(8)$ if and only if each other entry of~$A$ is strictly less than the corresponding entry of~$CSR[A_1]$.

\end{example}

For the transient of critical rows or columns, we get:

\begin{theorem}\label{t:CritColWiel}
Let $A\in\Rpnn$.
Then the transient of the critical rows and columns of~$A$ is equal
to~$\wiel(n)$
if and only if
it is of the form
$A=A_1\oplus A_2$ such that the index of $\digr(A_1)$
is~$\wiel(n)$, $A_1$ has a critical Hamiltonian cycle,
and~$A_2<CSR[A_1]$.
\end{theorem}
In contrast with the previous case, the critical graph need not have index~$\wiel(n)$.
It can also be a Hamiltonian cycle, and only $\digr(A_1)$ has index~$\wiel(n)$.

\subsection{Overview of Proofs}\label{sec:PfOverview}
Each of the next section is devoted to the proof of one of the main theorems 
(Theorem~\ref{t:mainres} and~\ref{t:wiel}). The proof's strategy is the same for both theorems.
It relies on the classical interpretation of entries of powers as weight of walks
(cf. Equations~\eqref{e:At} and~\eqref{e:representation}).

The necessity of the conditions is proved in 3 steps~:
\begin{itemize}
 \item Firstly, we use (slight refinements of) the results of~\cite{MNS} to prove properties of the critical graph.
 (Proposition~\ref{p:critgraph})
 \item Then we introduce a special type of walks with given start and end nodes, which we call 'interesting walks'. Loosely speaking, these walks are optimal in terms of both weight and length (Definition~\ref{def:interesting}). The structure of interesting walks is then studied in  Propositions~\ref{p:twiceopt}, \ref{p:W0} and~\ref{p:wielwalk}. In particular, interesting walks contain a Hamiltonian cycle that gives the renumbering of nodes up to where to start. The start is determined by the critical nodes.
 \item Finally, we use our results on the structure of interesting walks to show that if one condition is not fulfilled, then one can build new shorter walks and thus the interesting walk would not be optimal, either in terms of weight, or length.
\end{itemize}

The proof of the sufficiency is based on the fact that the index of $\digr(A_1)$ reaches the bound and the following perturbation lemma,
which will also be used in the proof of necessity and we think is interesting for its own sake.
We write $A<B$ if for all $i,j$, $a_{i,j}<b_{i,j}$ or $a_{i,j}=b_{i,j}=-\infty$.
\begin{lemma}[Perturbation lemma]
\label{l:perturbation}

If $A=A_1\oplus A_2$ with $A_1,A_2\in\Rpnn$ such that $A_2<CSR[A_1]$,
then, we have $\lambda(A)=\lambda(A_1)$, $\crit(A)=\crit(A_1)$,  $CS^tR[A]=CS^tR[A_1]$ for all $t$ and $T_1(A)=T_1(A_1)$.
\end{lemma}

This lemma will be proved, together with all other statements, in the next section. 

\section{Matrices attaining the Dulmage-Mendelsohn Bound}\label{sec:dm}

This section will be devoted to the proof of Theorem~\ref{t:mainres}.
In the rest of the work we assume $\lambda(A)=\1$
(because neither $T_1(A)$, nor the conditions of the theorem are modified if~$A$ is multiplied by a scalar) and set
$g:=\g(\crit(A))$.

Observe that $\lambda(A)=\1$ ensures that all closed walk have nonpositive weight,
so that removing a closed subwalk from a given walk can only increase its weight. A fact that we will use extensively in this paper.

\subsection{Perturbation Lemmas}

In this section, we prove two lemmas, to be used in both ways of the equivalence.
The first one was announced at the end of section~\ref{sec:thms}.

If for two matrices $A=(a_{i,j})_{i,j=1}^n$ and $B=(b_{i,j})_{i,j=1}^n$ we have $a_{i,j}\leq b_{i,j}$ 
for all $i,j$ and that
$a_{i,j}=b_{i,j}\Rightarrow a_{i,j}=b_{i,j}=-\infty$ then we write $A<B$. Observe that if $A<B$ and $C\le D$
then $A\otimes C<B\otimes D$ and~ $C\otimes A<D\otimes B$.

\begin{proof}[Proof of Perturbation Lemma~\ref{l:perturbation}]

If  $\lambda(A_1)=\0$, then $A_2=\0$, $A=A_1$ and there is nothing to prove. Otherwise, 
since $A_2<CSR[A_1]$ is invariant under $A\mapsto \lambda\otimes A$, we assume without loss of generality that $\lambda(A_1)=\1$.

To prove the first equalities, consider a diagonal matrix~$D$ that provides a visualization
scaling for $A_1$. Then we obtain
$D^{-}AD=D^{-}A_1D\oplus D^{-}A_2D$
where $D^{-}A_2D<D^{-}CSR[A_1]D\leq \1$. So $D^{-}AD$ is visualized and $(D^{-}AD)_{i,j}=\1$ if and only if 
 $(D^{-}A_1D)_{i,j}=\1$, which implies that $\lambda(A)=\1=\lambda(A_1)$ and $\crit(A)=\crit(A_1)$, which are the 
first two equalities in the statement of Lemma~\ref{l:perturbation}.

To prove the remaining two equalities, let us first prove the following statement by induction:
\begin{equation}
\label{e:induction}
A^t=A_1^t\oplus R_t,\quad \text{where}\quad R_t<CS^tR[A_1]. 
\end{equation}
For $t=1$, set~$R_1=A_2$.

Suppose that~\eqref{e:induction} holds.
Let us prove that it also holds when $t$ is
replaced by $t+1$. We have
\begin{equation}
\label{e:induction1}
A^{t+1}=(A_1\oplus R_1)(A_1^t\oplus R_t)=
A_1^{t+1}\oplus R_1 A_1^t\oplus A_1 R_t\oplus R_1R_t.
\end{equation}
We bound from above the last three terms on the right-hand side of~\eqref{e:induction1}. 

We have:
\begin{itemize}
\item[1.] $R_1A_1^t<CSR[A_1]A_1^t= CS^{t+1}R[A_1]$. 
by Proposition~\ref{p:CSRprops} part~\eqref{CSR-pseudoGpe}
\item[2.] $A_1R_t<A_1CS^tR[A_1]= CS^{t+1}R[A_1]$, by Proposition~\ref{p:CSRprops} part~\eqref{CSR-pseudoGpe}.
\item[3.] $R_1R_t< CSR[A_1] CS^tR[A_1]=CS^{t+1}R[A_1]$, by Proposition~\ref{p:CSRprops} part~\eqref{CSR-group}.
\end{itemize}

Thus $A^{t+1}= A_1^{t+1}\oplus R_{t+1}$ where $R_{t+1}= R_1 A_1^t\oplus A_1 R_t\oplus R_1R_t$ satisfies
$R_{t+1}<CS^{t+1}R$.

Observe that $A^t=A_1^t$ for $t\geq T_1(A_1)$. Indeed, for any such $t$,
\begin{equation*}
\begin{split}
A^t&=A_1^t\oplus R_t
=CS^tR[A_1]\oplus B^t[A_1]\oplus R_t\\
&= CS^tR[A_1]\oplus B^t[A_1]=A_1^t. 
\end{split}
\end{equation*}

From this equality, the periodicity of $CSR$ and 
the observation that 
$\lim_{k\to\infty} ((B[A_1])^{t+k\sigma})=\0$ since $\lambda(B[A_1])<\1$, 
we deduce that $A$ and $A_1$ have the same $CSR$ since
$$ CS^tR[A]=\lim_{k\to\infty} A^{t+k\sigma}=\lim_{k\to\infty} A_1^{t+k\sigma}=CS^tR[A_1],$$
where $\sigma$ is the cyclicity of $\crit(A)$. Thus $CS^tR[A]=CS^tR[A_1]$ for all $t$. 

To prove the remaining equality $T_1(A)=T_1(A_1)$, we first observe that 
$(B^t[A])_{i,j}\leq (A^t)_{i,j}\leq (CS^tR[A]\oplus B^t[A])_{i,j}$, and hence
\begin{equation*}
(A^t)_{i,j}\neq (CS^tR[A]\oplus B^t[A])_{i,j}\Leftrightarrow 
(A^t)_{i,j}<(CS^tR[A])_{i,j}.
\end{equation*}
Next we use that $(B^t[A_1])_{i,j}\leq (A_1^t)_{i,j}\leq (CS^tR[A_1]\oplus B^t[A_1])_{i,j}$, that
 $A^t=A_1^t\oplus R_t$ where $R_t<CS^tR[A_1]$, and that
$CS^tR[A_1]=CS^tR[A]$, so that we have the following equivalences: 
\begin{equation*}
\begin{split}
& (A^t)_{i,j}\neq (CS^tR[A]\oplus B^t[A])_{i,j}\Leftrightarrow
(A^t)_{i,j}<(CS^tR[A])_{i,j}\Leftrightarrow
(A_1^t)_{i,j}\oplus (R_t)_{i,j}<(CS^tR[A_1])_{i,j}\\
&\Leftrightarrow (A_1^t)_{i,j}<(CS^tR[A_1])_{i,j}
\Leftrightarrow (A_1^t)_{i,j}\neq (CS^tR[A_1]\oplus B^t[A_1])_{i,j},
\end{split}
\end{equation*}
thus $T_1(A)=T_1(A_1)$.
\end{proof}


\begin{lemma}
\label{l:CSRA=CSRAB}
Let~$A_1,B_1, A_2$ be defined from the same matrix~$A$ by Equations~\eqref{e:A1},~\eqref{e:B1} and \eqref{e:A2}.
Assume that Conditions~1 to~5 of  Theorem~\ref{t:mainres} are satisfied, that is 
$\crit(A)$ is strongly connected and contains all edges of $1\cdots g1$, $g$ and $n$ are coprime, $A_2<CSR[A_1]$
and~$B_1$ satisfies Condition~\ref{CB1}.
Then $CS^tR[A_1]=CS^tR[A_1\oplus B_1]$ for all $t$.
\end{lemma}
\begin{proof}


As in the first paragraph of the proof of Lemma~\ref{l:perturbation}
consider a diagonal matrix providing a visualization scaling for 
$A_1$, assuming without loss of generality that $\lambda(A_1)=\1$. We then have $D^{-}A_2D<D^{-}CSR[A_1]D\leq\1$,
and this shows that all arcs of $\digr(A_2)$ are non-critical, hence $\crit(A)=\crit(A_1\oplus B_1)$.
In particular, $\crit(A_1\oplus B_1)$ is strongly connected and contains all edges of $1\cdots g1$. The same is true about
$\crit(A_1)$, and we also have $\lambda(A_1\oplus B_1)=\1$.

Recasting the above
statement about CSR in terms of walks with Proposition~\ref{p:representation} we have to prove for any $k,l\in\{1,\ldots,n\}$ that
\begin{equation*}
\max\{p(W)\colon W\in\walkslennode{k}{l}{t,g}{Z_0}[A_1]\}=
\max\{p(W)\colon W\in\walkslennode{k}{l}{t,g}{Z_0}[A_1\oplus B_1]\},
\end{equation*}
Since $A_1\leq A_1\oplus B_1$, we have
the inequality
\begin{equation*}
\max\{p(W)\colon W\in\walkslennode{k}{l}{t,g}{Z_0}[A_1]\}\leq 
\max\{p(W)\colon W\in\walkslennode{k}{l}{t,g}{Z_0}[A_1\oplus B_1])\},
\end{equation*}
To prove the opposite, we need to take an arbitrary walk 
$W\in\walkslennode{k}{l}{t,g}{Z_0}[A_1\oplus B_1])$ and prove that there exists a walk
$W'\in\walkslennode{k}{l}{t,g}{Z_0}[A_1])$ such that $p(W)\leq p(W')$.

There are two kinds of arcs in $W$ that may not be in $\digr(A_1)$.\\

A. $ij$ with $j>i+1>g+1$ and $j\equiv_g i+1$. In this case, 
we can replace $ij$ by the path $i(i+1)\cdots j$.
The resulting walk visits even more nodes and hence will go through a node of $\crit(A_1)$.
It has the same length modulo $g$. Due to Condition~\ref{CB1} of Theorem~\ref{t:mainres}, its weight
is not smaller than $p(W)$. Thus we can assume that $W$ does not contain such arcs.

B. $ij$ with $j<i$ and $j\equiv_g i+1$. 

Since we assumed that~$W$ contains no arc $kl$ s.t. $l>k+1$, the only arc that go from a node to a larger node are of type $k,k+1$,
thus if a subwalk of $W$ goes from~$i$ to~$j>i$ it goes through all nodes numbered between~$i$ and~$j$.

If $W$~goes to~$Z_0$ after arc~$ij$, it has to go through~$i$ again before reaching~$Z_0$, because the only arc to reach~$Z_0$
from~$\digr(B_1)$ is $n1$. In this case, define $W_1$ as the closed subwalk that starts with the arc~$ij$ and follows~$W$ until it goes back to~$i$.

If $W$~does not go to~$Z_0$ after arc~$ij$, it has to come from~$Z_0$ before, so it has to go through~$g,g+1$ which is the only arc leaving~$Z_0$.
So, it has been in~$j$ before reaching~$i$ and arc~$ij$.
Then, define $W_1$ as the closed subwalk that starts with the last occurrence of~$j$ before arc~$ij$ and follows~$W$ until it goes back to~$i$. 

In both cases, $W_1$ lives on~$\digr(B_1)$ so its length is divisible by~$g$ and it can be removed from~$W$.

The resulting walk has the same length modulo $g$, goes through a node of $Z_0$ and has larger weight than $p(W)$.\\ 

Iterating the process, we build the~$W'$ we are looking for.
\end{proof}



\subsection{Proof of Sufficiency}
Let us now prove that Conditions 1--6 imply~$T_1(A)=\DM(g,n)$.
We assume that the conditions are satisfied.

By Lemmas~\ref{l:perturbation} and~\ref{l:CSRA=CSRAB}, we  have $CSR[A]=CSR[A_1\oplus B_1]=CSR[A_1]$.

As $A_2<CSR[A_1\oplus B_1]$ (Condition~\ref{CA2} and 
Lemma~\ref{l:CSRA=CSRAB}), by Lemma~\ref{l:perturbation} we have $T_1(A)=T_1(A_1\oplus B_1)$ so 
we can assume that $A_2=\bf{0}$ and we do it from now on.

Entry $A^{\DM(g,n)-1}_{g+1,n}$ is the largest weight of a walk $W$ from $g+1$ to $n$ with length $\DM(g,n)-1$.
Let us prove $p(W)<(CS^{n-1}R)_{g+1,n}[A]$, which ensures that~$T_1(A,B)\ge \DM(g,n)$, 
since in this case  $A^t_{g+1,n}<(CS^tR)[A]\oplus B^t[A])_{g+1,n}$ for 
$t=\DM(g,n)-1$ (also recalling that $(CS^tR)[A]=(CS^{n-1}R)[A]$ by the periodicity
of $CSR$).
The other inequality follows from Theorem~\ref{t:bounds}.

{\bf Case 1.}
 If $W$ does not go through $Z_0$, then it is a walk on $\digr(B_1)$ and $p(W)= (B_1)_{g+1,n}^{\DM(g,n)-1}$.
Using Condition~\ref{CB2}, we conclude that $p(W)<(CS^{n-1}R)_{g+1,n}[A]$.

{\bf Case 2.} 
 Assume now that $W$ goes through $Z_0$ and contains an arc $(i,j)$ such that 
 $ j>i+1$, $ i>g$ and $j\equiv_g (i+1)$. Then we can replace this arc by the path $i(i+1)\cdots j$ thus obtaining a new walk $W'$.
Using Condition~\ref{CB1}, we conclude that $p(W)<p(W')$. But, we also have $p(W')\leq (CS^{n-1}R)_{g+1,n}[A]$, since $W'$
visits a node of $Z_0$ and has length $n-1$ modulo $g$. Thus $p(W)<p(W')\leq  (CS^{n-1}R)_{g+1,n}[A]$.  

{\bf Case 3.}
 Assume that $W$ goes through $Z_0$ and does not contain an arc $(i,j)$ such that 
$ j>i+1$, $ i>g$ and $j\equiv_g (i+1)$. 
Then $W$ can be decomposed into a path from 
$g+1$ to $n$, and some cycles.
Since we assumed that $A_2=\0$ and $W$ 
contains no arc with $j>i+1>g$,
the only path from~$g+1$ to~$n$ is the path that follows the numbers and has length $n-1-g$,
so that the total length of the cycles is  $\DM(g,n)-1- (n-1-g)=g(n-1)$.

Since $W$ goes through a node of $Z_0$, and $1\cdots n1$ is the only cycle that connects~$Z_0$ and $\digr(B_1)$,
the cycle decomposition of $W$ contains at least one copy of
the cycle $1\cdots n1$. 
But, again since $A_2=\0$, $1\cdots n1$ is the only cycle whose length~$n$ is not a 
multiple of $g$. As this length is, moreover, coprime with $g$,
the walk should contain at least $g$ copies of the cycle $1\cdots n1$. But then their total length is $gn>g(n-1)$, a contradiction.
Hence this case is impossible, and the attainment of Dulmage-Mendelsohn bound has been proved for all 
possible cases.

\subsection{Critical Graph}
The end of the section will be devoted to the proof of the necessity of Conditions 1--6,
along the lines presented in Section~\ref{sec:PfOverview}.

In this subsection, we prove the following proposition.

\begin{proposition}
\label{p:critgraph}
If $T_1(A)=\DM(g,n)$, then $\crit(A)$ is strongly connected and contains only one cycle of length $g$ up to choice of its first node.
\end{proposition}

To prove Proposition~\ref{p:critgraph}, we will use techniques from~\cite{MNS} related to CSR expansions
and walks. We will first recall the main statements that will be required, and then the proof of 
Proposition~\ref{p:critgraph} will be given after the statement of Proposition~\ref{p:TcrToT1}.

Proposition~\ref{p:representation} below is an extended version of Proposition~\ref{p:CSRprops} part~(i).
In particular, it builds on the idea that the CSR terms in Definition~\ref{def:CSR} 
can be defined using any completely reducible subgraph of $\crit(A)$ instead of the full $\crit(A)$.

\begin{proposition}[cf. \cite{MNS}, Theorem 6.1]
\label{p:representation}
Let~$A\in\Rpnn$ be a matrix with~\mbox{$\lambda(A)=\1$} and 
$C,S$ and $R$ be
the CSR terms of~$A$ with respect to some completely reducible
subgraph~$\subcrit$ of the critical graph~$\crit(A)$.

Let~$\gamma$ be a multiple of~$\gamma(\subcrit)$ 
and~$\mN$ a set of some
nodes of $\subcrit$ that contains one node of every s.c.c.\ of~$\subcrit$.

Then we have, for any~$i,j$ and~$t\in\Nat$:
\begin{equation}\label{e:representation}
(CS^tR)_{i,j}=p\left(\walkslennode{i}{j}{t,\gamma}{\mN}\right)
\end{equation}
where $\walkslennode{i}{j}{t,\gamma}{\mN}:=\left\{W\in\walksnode{i}{j}{\mN}
\,\big|\,l(W)\equiv t\pmod\gamma\right\}$
\end{proposition}

In particular, CSR terms can be defined using a s.c.c. of $\crit(A)$ rather than the 
whole $\crit(A)$, 
since any s.c.c. of $\crit(A)$ is a completely reducible subgraph of
$\crit(A)$. 
Following~\cite{MNS} let 
$\subcrit_1,\ldots,\subcrit_l$ be the s.c.c.'s of $\crit(A)$ with
node sets $N_1,\ldots N_l$, and let $C_{\subcrit_1}$,
$S_{\subcrit_1}$, $R_{\subcrit_1}$ be the CSR terms defined with
respect to $\subcrit_1$. Let $A^{(1)}=A$, and for $\nu=2,\ldots,l$ define a
matrix $A^{(\nu)}$ by setting the entries of $A$ with
rows and columns in $N_1\cup\ldots\cup N_{\nu-1}$ to $\0$, and let
$C_{\subcrit_\nu}$, $S_{\subcrit_\nu}$, $R_{\subcrit_\nu}$ be the
CSR terms defined with respect to $\subcrit_\nu$ in $A^{(\nu)}$. 
By the {\em dimension}
of $A^{(\nu)}$ we will mean the number of elements in $N\backslash (N_1\cup\ldots\cup N_{\nu-1})$.

\begin{proposition}[\cite{MNS}, Corollary 6.3]
\label{p:representationSCC}
If $\subcrit_1,\cdots,\subcrit_l$ are the s.c.c.'s of~$\crit(A)$,
then we have:
\begin{equation}\label{e:early-exp}
C S^t R=\bigoplus_{\nu=1}^lC_{\subcrit_\nu} S_{\subcrit_\nu}^t R_{\subcrit_\nu}.
\end{equation}
\end{proposition}

Let us now prove the following bound on $T_1(A)$:

\begin{proposition}
\label{p:T1ABN}
Let $n_{\nu}$ for $\nu=1,\ldots,l$ be the dimension of $A^{(\nu)}$,
and $g_{\nu}$ for $\nu=1,\ldots,l$ be the girth of $\crit_{\nu}$. Then
\begin{equation}
T_1(A)\leq\max_{\nu=1,\ldots,l} \DM(g_{\nu},n_{\nu}).
\end{equation}
\end{proposition}

\begin{proof}
For each $\nu=1,\ldots,l$, the weak CSR expansion applied to~$A^{(\nu)}$ reads
\begin{equation}
\label{e:CSR-nu}
(A^{(\nu)})^t=C_{\crit_{\nu}}(S_{\crit_{\nu}})^t R_{\crit_{\nu}}\oplus (A^{(\nu+1)})^t,\quad t\geq T.
\end{equation}
The smallest $T$ for which the weak CSR expansion~\eqref{e:CSR-nu} holds is bounded by 
$\DM(g_{\nu},n_{\nu})$ by Theorem~\ref{t:bounds}.
If $t\geq \max_{\nu=1,\ldots,l} \DM(g_{\nu},n_{\nu})$ then~\eqref{e:CSR-nu} holds for all $\nu=1,\ldots l$ 
and by successive replacement, we have:
\begin{equation}
\label{e:fullCSR}
A^t=\bigoplus_{\nu=1}^l C_{\crit_{\nu}}(S_{\crit_{\nu}})^tR_{\crit_{\nu}}\oplus (A^{(l+1)})^t,
\end{equation}
Observing that $A^{(l+1)}=B$ and that the CSR terms sum up to $CS^tR$ by 
Proposition~\ref{p:representationSCC} we see that~\eqref{e:fullCSR} is exactly the 
weak CSR expansion $A^t=CS^tR\oplus B^t$.
\end{proof}

Bounds for $T_1(A)$ in~\cite{MNS} are based on the concept of the cycle removal
threshold defined as follows.

\begin{definition}
\label{def:Tcr}
Let~$\subcrit$ be a subgraph of~$\digr(A)$ and~$\gamma\in\Nat$.\\
The {\em cycle removal threshold}~$T_{cr}^\gamma(\subcrit)$ of~$\subcrit$ is
the smallest nonnegative integer~$T$ for which the following
holds: for all walks~$W\in\walksnode{i}{j}{\subcrit}$ with
length~$\geq T$, there is a walk
$V\in\walksnode{i}{j}{\subcrit}$ obtained from~$W$ by removing
cycles and possible inserting cycles
of~$\subcrit$ such that $l(V)\equiv l(W) \pmod{\gamma}$ and
$l(V)\le T$.
\end{definition}

The idea behind this definition is to be able to shorten
an optimal walk while keeping it optimal and keeping its length modulo~$\gamma$, thus proving inequalities between $CS^tR$ and~$A^t$.

The following proposition is stated in~\cite{MNS} and
proved there by ``arithmetical method".

\begin{proposition}[\cite{MNS}, Proposition 9.5]
\label{p:TcRLin}
For $A\in\Rpnn$ and $\subcrit$ a subgraph of~$\digr(A)$ with~$n'$ nodes, we have:
\begin{equation}
\label{e:TcRLin}
\forall \gamma\in\Nat, T_{cr}^\gamma(\subcrit)\le \gamma n +n-n'-1.
\end{equation}
\end{proposition}

The next proposition can be proved using a slight generalization
of~\cite{MNS}, Proposition 6.5\;(i),
which differs in the fact that~$\subcrit$ is considered as a whole and not each
s.c.c.\ at a time.
The proof of~\cite{MNS} actually shows this stronger statement.

\begin{proposition}
\label{p:TcrToT1}
Let $A$ be a square matrix such that all s.c.c.'s of $\crit(A)$,
have the same girth $g$. Let $\subcrit$ be the subgraph of $\crit(A)$ consisting
of all cycles of length $g$.
Then $$T_1(A)\le T_{cr}^{g}(\subcrit)-g+1.$$
\end{proposition}

%
%
%

We are finally able to prove Proposition~\ref{p:critgraph}.
\begin{proof}[{\bf Proof of Proposition~\ref{p:critgraph}}]
Let us assume that $T_1(A)=\DM(g,n)$. We want to prove that $\crit(A)$ is strongly connected and has only one cycle with length~$g$,
up to choice of its first node.

To apply Proposition~\ref{p:TcrToT1} we need to show that all s.c.c's of $\crit(A)$ have girth~$g$.

Otherwise, w.r.t. the notation of Proposition~\ref{p:T1ABN} let $\crit_1$ be a component with girth~$g_1<g$.
We have $\DM(g_1,n_1)<\DM(g,n)$ (recall that $\DM(g,n)=g(n-2)+n$). 
For all other components of $\crit(A)$ we have 
$\DM(g_{\nu},n_{\nu})<\DM(g,n)$ since $n_{\nu}<n$. Therefore using the bound of 
Proposition~\ref{p:T1ABN} we would have $T_1(A)<\DM(g,n)$.

Now, we will combine Propositions~\ref{p:TcrToT1} and~\ref{p:TcRLin}
to show both the connectivity of the $\crit(A)$ and the uniqueness of the shortest critical cycle.
Let us denote by~$n'$ the number of nodes of the graph~$\subcrit$ of Proposition~\ref{p:TcrToT1} and 
set~$\gamma=g$ 
in~\eqref{e:TcRLin} (see Proposition~\ref{p:TcRLin}).
Then we have 
$$T_1(A)\le T_{cr}^{g}(\subcrit)-g+1
\le (gn+n-n'-1) -g+1
= \DM(g,n)+g-n'.$$

If $\crit(A)$ is not strongly connected then~$\subcrit$ is not and $n'\ge 2g$. If $\crit(A)$  is
strongly connected but contains more than one cycle, then $n'>g$.
Indeed, one can not have two critical cycles of length~$g$ with the same set of nodes,
because it would build a shorter cycle so that~$g$ would not be the girth anymore.
Thus, in both cases, we would have $n'>g$ and $T_1(A)<\DM(g,n)$.

Finally, $\crit(A)$ is strongly connected and contains exactly one cycle of length~$g$, up to choice of its first node.
\end{proof}

\subsection{The Interesting Walk and Its Structure}
In this section we assume that $T_1(A)=\DM(g,n)$.
and deduce from this the
structure of special walks which we call interesting.

By Proposition~\ref{p:critgraph}, there is a unique critical cycle of length~$g$ up to choice of its first node. 
The subgraph consisting of all nodes and edges of this cycle will be denoted by $Z_0$.

\begin{definition}\label{def:interesting}
A walk $W\in \walkslennode{i}{j}{t}{Z_0}$ is called {\em twice optimal} if it has minimal length among 
all the walks with maximal weight in the set $\walkslennode{i}{j}{t,g}{Z_0}$.
It is called {\em interesting} if it is twice optimal and has length~$\DM(g,n) +g-1$.
\end{definition}

Interesting walks are twice optimal with maximal possible length among all entries and matrices.
Their particular structure, described in Proposition~\ref{p:W0} will define matrices~$A_1, A_2, B_1$.

\begin{proposition}
\label{p:twiceopt}
If $T_1(A)=\DM(g,n)$, then there exists $(i,j)$ such that $(A^{\DM(g,n)-1})_{i,j}<(CS^{\DM(g,n)-1}R)_{i,j}$. 
For any such~$(i,j)$ there is an interesting walk from~$i$ to~$j$. 
\end{proposition}
\begin{proof}
We set $t=\DM(g,n)-1$. 

If there is no~$(i,j)$ such that~$(A^{\DM(g,n)-1})_{i,j}<(CS^{\DM(g,n)-1}R)_{i,j}$, then for all~$i,j$ we have $(A^t)_{i,j}\ge (CS^tR)_{i,j}$.

By definition of~$B$, we have $(A^t)_{i,j}\ge (B^t)_{i,j} $ for all~$t,i,j$.
We hence have $(A^t)_{i,j} \ge (B^t)_{i,j}\oplus (CS^tR)_{i,j}$.
We always have the opposite inequality $(A^t)_{i,j} \le (B^t)_{i,j}\oplus (CS^tR)_{i,j}$
by distinguishing whether a walk visits the critical graph or not.
Hence $(A^t)_{i,j}=(B^t)_{i,j}\oplus (CS^tR)_{i,j}$.
But this means $T_1(A) \leq t <\DM(g,n)$, which contradicts $T_1(A) = \DM(g,n)$.

Let us prove the second part of the proposition.


Proposition~\ref{p:TcRLin}, applied with~$\subcrit=Z_0$, $\gamma=g$ 
and $n_1=g$, 
implies that twice optimal walks have length at most 
$t+g$ (alternatively, see the proof of \cite{MNS} Theorem 4.1).

Now, if there is no interesting walk from~$i$ to~$j$ it means that the set $\walkslennode{i}{j}{t,g}{Z_0}$
contains a walk with optimal weight and length strictly less than~$t+g$. 
However, this length 
is congruent to $t$ modulo $g$, hence it is less than or equal to $t$ and, furthermore,
can be made equal to~$t$ by inserting copies of~$Z_0$.
The weight of this walk is $(CS^{\DM(g,n)-1}R)_{i,j}$ by Proposition~\ref{p:representation},  so $(A^{\DM(g,n)-1})_{i,j}\geq 
(CS^{\DM(g,n)-1}R)_{i,j}$.
\end{proof}


\begin{proposition}
\label{p:n=g=2}
If $n=g=2$ and $A\in\Rp^{2\times 2}$,
then \mbox{$T_1(A)=\DM(2,2)=2$}
if and only if $a_{11}\neq a_{22}$. 
\end{proposition}
\begin{proof}
 Observe that $n=g=2$ implies that $\crit(A)$ consists of the nodes and arcs 
of the unique cycle of length $2$ up to choice of its first node, and that both nodes of $\digr(A)$ are critical.
Thus $B_N=\0$, $T_1(A)=T(A)$ and $T_1(A)<2 \Leftrightarrow A^3=A$.

Let us notice that 
$(a_{11})^{\otimes 2} <a_{12}\otimes a_{21}=\lambda(A)=\1$ and $a_{22}^{\otimes 2}<a_{12}\otimes a_{21}=\lambda(A)=\1$,
and compute~$A^3$.

Consider first the off-diagonal entries. In this case we have
\begin{equation*}
(A^3)_{k,l}= \max\{a_{k,l}a_{l,k}a_{k,l}, (a_{k,k})^2 a_{k,l}, a_{k,l}(a_{l,l})^2\}= a_{k,l}
\end{equation*}
for any such $k,l\in\{1,2\}$.

Consider now
\begin{equation*}
\begin{split}
(A^3)_{k,k}&=\max \{(a_{k,k})^3, a_{k,l}a_{l,l}a_{l,k}, a_{k,l}a_{k,k}a_{l,k}\}\\
 &= \max\{a_{k,k},a_{l,l}\},
\end{split}
\end{equation*}
for $k,l\in\{1,2\}$ and 
$l\neq k$.
This is not equal to $a_{k,k}$ if and only if $a_{k,k}<a_{l,l}$. 
Finally, $A^3=A\Leftrightarrow a_{k,k}=\max \{a_{k,k},a_{l,l}\}\Leftrightarrow a_{1,1}=a{2,2}$ and the proof is complete.
\end{proof}

The following proposition, to be proved in Subsection~\ref{ss:ProofpW0},
shows uniqueness and characterizes the interesting walk in the remaining cases of~$g$ and~$n$.

\begin{proposition}
\label{p:W0}
Let $A\in\Rpnn$ be such that $g=\g(\crit(A))\ge 2$,
$T_1(A)=\DM(g,n)$, and
not $n=g=2$. For any interesting walk~$W_0$, there is a renumbering of the nodes such that $1,\ldots,g$ are the nodes of the (only) 
critical cycle of length $g$ and
\begin{equation}
\label{e:W0}
W_0=(g+1)\ldots n (1 \dots n)^g,
\end{equation}
where $(1\dots n)^g=\overbrace{(1\dots n)(1\dots n)\dots (1\dots n)}^{g\text{ times}}$.\\
\end{proposition}

\begin{corollary}
\label{c:hamiltonian}
Under the conditions of Proposition~\ref{p:W0}, $\digr(A)$
has an Hamiltonian cycle with maximal weight, unique  up to choice of its first node,
which is labeled $1\dots n1$ by the renumbering stated in Proposition~\ref{p:W0}.
\end{corollary}
\begin{proof}
By contradiction, suppose that there is a different Hamiltonian cycle with the largest weight.
It can replace one of the copies of $1\ldots n1$ in $W_0$, and the walk should still be interesting. 
However as $g\geq 2$, the resulting walk contains edges of at least two different Hamiltonian cycles,
so it cannot be represented as~\eqref{e:W0}, 
which is in contradiction with this walk being interesting.
\end{proof}


\begin{remark}
\label{r:W0}
By Proposition~\ref{p:W0} and Corollary~\ref{c:hamiltonian}, 
the Hamiltonian cycle with maximal weight induces a renumbering of the nodes, which is unique up to choice of the first node.
Then, the first node is defined as the end of the only edge of 
$Z_0$ that does not belong to the Hamiltonian cycle.
Thus, the renumbering is unique and so is the interesting walk.
\end{remark}

\subsection{Proof of Proposition~\ref{p:W0}}
\label{ss:ProofpW0}

Let us first note that the case $g=n$ is impossible unless $n=g=2$. Indeed, if $n=g>2$
then $\DM(g,n)=n(n-2)+n=n(n-1)>(n-1)^2+1$, 
which is the Wielandt bound for the 
periodicity transient, so in this case $\DM(g,n)$ cannot be attained. The case $n=g=2$ has been considered in
Proposition~\ref{p:n=g=2}.

The following elementary
number-theoretic lemma will be especially useful in what follows.

\begin{lemma}
\label{l:HA}
Let $a_1,\ldots, a_{s}\in\Z$. Then there is a nonempty subset 
$I\subseteq\{1,\ldots,s\}$ with 
$\sum_{i\in I} a_i\equiv_s 0$.
\end{lemma}
This lemma will allow us to remove some cycles from a walk an keep its length modulo~$s$ as soon as we have~$s$ cycles  that do not intersect in the walk.

The first step of the proof 
of Proposition~\ref{p:W0} is to 
establish properties of the structure of interesting walks.

\begin{lemma}
\label{l:leqg}
In any interesting walk~$W_0$, there are exactly $g$ occurrences of each node of $Z_0$ 
and exactly $g+1$ occurrences of each  node not in $Z_0$.
\end{lemma}
\begin{proof}
This is an improvement on the proof of Theorem~\ref{t:bounds} in~\cite{MNS} with extra care on the extremal cases.
Let us first argue that the number of occurrences does not exceed $g$ and $g+1$, respectively.

By contradiction, let $l\in Z_0$ occur $k\geq g+1$ times, then we have $W_0=V_0lV_1l\ldots 
lV_{k-1}lV_k$
where $l$ occurs in no $V_i$.

We have $k-1\geq g$ and by Lemma~\ref{l:HA} some of the cycles
$lV_{p}l$ (for $p=1,\ldots,k-1$) can be removed in such a way that the resulting walk has 
the same length modulo $g$. 
Moreover, the resulting walk still goes through a node of $Z_0$ (namely~$l$) and has the same length modulo $g$
meaning that $W_0$ is not twice optimal. 

Let $m\notin Z_0$ and decompose $W_0=W_1sW_2$ so that $W_1$ contains only nodes not in $Z_0$ and 
$s\in Z_0$. Then we have two cases:\\
a) One of the walks $W_1$ or $W_2$ does not contain $m$. Then the remaining walk can have
no more than $g$ occurrences of $m$, otherwise these occurrences lead to at least $g$ cycles some of which can be removed in such a way that the resulting walk
has the same length modulo $g$ and goes through a node of $Z_0$, contradicting the optimality of
$W_0$ (the weight of the resulting walk is also not smaller since by $\lambda(A)=\1$ the weight of each 
cycle is not bigger than $\1$).\\ 
b) Both $W_1$ and $W_2$ contain $m$ at least once. If there are at least two occurrences of 
$m$ in $W_1$ then the cycle between these occurrences can be moved to $W_2$. Hence we can assume
that $W_1$ contains $m$ no more than once. As in a), $m$ can occur in $W_2$ no more than $g$ times,
and the total number of $m$'s occurrences is thus bounded by $g+1$.

The total number of occurrences of all nodes in $W_0$ is thus bounded from above by $g^2+(n-g)(g+1)=n-g+ng$.
Observe now that the total number of these occurrences is exactly $g(n-1)+n=n-g+ng$, since the length of 
$W_0$ is $\DM(g,n)+g-1=g(n-1)+n-1$.  Hence each node in $Z_0$ occurs exactly $g$ times and each node not in $Z_0$ exactly $g+1$ times.
\end{proof}

\begin{lemma}[Interlacing]
\label{l:interlace}
Let $i\in Z_0$ and $j\notin Z_0$.
\begin{itemize}
\item[(i)] In any interesting walk, there is exactly one occurrence of node $j$ between two consecutive occurrences of~$i$.
\item[(ii)] In any interesting walk, there is exactly one occurrence of $j$ before the first and exactly one after the last occurrence of~$i$.
\end{itemize}
\end{lemma}
\begin{proof}
%
We are going to prove that, for each $k\colon 0\leq k<g$, there are exactly
$k+1$ occurrences of $j$ before the $(k+1)$-th occurrence of $i$ and $g-k$
occurrences of $j$ after that occurrence of $i$. This implies both parts of the lemma.

We first show after  $k+1$ occurrences of $i$ we have no more than $g-k$ occurrences of $j$, for otherwise 
we have $k$ consecutive closed walks going through $i$ and at least $g-k$ consecutive closed walks going through $j$.
Figure~\ref{fig:flower} depicts the situation.
Thus the overall number of closed walks is at least $g$ and by Lemma~\ref{l:HA} some of these closed walks can be removed from the walk. The resulting walk still goes through a node of $Z_0$ (namely $i$) and has the same length modulo $g$, so 
$W_0$ is not twice optimal, a contradiction.

\begin{figure}
\centering
\begin{tikzpicture}[>=latex']
\node[circle,very thick,draw,minimum size=9mm] (n1) at (-5,0) {};
\node[circle,very thick,draw,minimum size=9mm] (ni) at (-2,0) {$i$};
\node[circle,very thick,draw,minimum size=9mm] (nj) at (2,0) {$j$};
\node[circle,very thick,draw,minimum size=9mm] (nf) at (5,0) {};

\draw[->,very thick] (n1) -- (ni);
\draw[->,very thick] (ni) -- (nj);
\draw[->,very thick] (nj) -- (nf);

\draw[->,very thick] (ni) .. controls +(145:4) and +(125:4) .. (ni) ;
\draw[->,very thick] (ni) .. controls +(125:4) and +(105:4) .. (ni) ;
\draw[->,very thick] (ni) .. controls +(105:4) and +(85:4) .. (ni) node [midway, above] {$k$ closed walks};
\draw[->,very thick] (ni) .. controls +(85:4) and +(65:4) .. (ni) ;

\draw[->,very thick] (nj) .. controls +(125:4) and +(105:4) .. (nj) ;
\draw[->,very thick] (nj) .. controls +(105:4) and +(85:4) .. (nj) ;
\draw[->,very thick] (nj) .. controls +(85:4) and +(65:4) .. (nj) node [midway, above] {$g-k$ closed walks};
\draw[->,very thick] (nj) .. controls +(65:4) and +(45:4) .. (nj) ;
\draw[->,very thick] (nj) .. controls +(45:4) and +(25:4) .. (nj) ;

\end{tikzpicture}
\caption{More than $g-k$ occurrences of $j$ after $k+1$ occurrences of $i$}
\label{fig:flower}
\end{figure}

Let us also show that we have no more than $k+1$ nodes of $j$ before the $(k+1)$-th occurrence of $i$. 
Indeed, after the $(k+1)$-th occurrence of $i$ we still have $g-k-1$ occurrences of~$i$ by Lemma~\ref{l:leqg}. If the hypothesis is not true 
then we have $g-k-1$ closed walks going through~$i$ and at least $k+1$ closed walks going through $j$. Figure~\ref{fig:flower2} depicts the situation.
 Thus we have at least $g$ 
closed walks in total and by Lemma~\ref{l:HA} some of these closed walks can be removed from the walk.  We conclude that $W_0$ is not twice optimal,
a contradiction.

\begin{figure}
\centering
\begin{tikzpicture}[>=latex']
\node[circle,very thick,draw,minimum size=9mm] (n1) at (-5,0) {};
\node[circle,very thick,draw,minimum size=9mm] (nj) at (-2,0) {$j$};
\node[circle,very thick,draw,minimum size=9mm] (ni) at (2,0) {$i$};
\node[circle,very thick,draw,minimum size=9mm] (nf) at (5,0) {};

\draw[->,very thick] (n1) -- (nj);
\draw[->,very thick] (nj) -- (ni);
\draw[->,very thick] (ni) -- (nf);

\draw[->,very thick] (ni) .. controls +(125:4) and +(105:4) .. (ni) ;
\draw[->,very thick] (ni) .. controls +(105:4) and +(85:4) .. (ni) ; 
\draw[->,very thick] (ni) .. controls +(85:4) and +(65:4) .. (ni) node [midway, above] {$g-k-1$ closed walks};
\draw[->,very thick] (ni) .. controls +(65:4) and +(45:4) .. (ni) ;

\draw[->,very thick] (nj) .. controls +(145:4) and +(125:4) .. (nj) ;
\draw[->,very thick] (nj) .. controls +(125:4) and +(105:4) .. (nj) ;
\draw[->,very thick] (nj) .. controls +(105:4) and +(85:4) .. (nj) node [midway, above] {$k+1$ closed walks};
\draw[->,very thick] (nj) .. controls +(85:4) and +(65:4) .. (nj) ;
\draw[->,very thick] (nj) .. controls +(65:4) and +(45:4) .. (nj) ;

\end{tikzpicture}
\caption{More than $k+1$ occurrences of $j$ before $(k+1)$-th occurrence of $i$}
\label{fig:flower2}
\end{figure}

However, the total number of occurrences of $j$ before and after the $(k+1)$-th occurrence of $i$ is $g+1$,
and therefore there are exactly $k+1$ occurrences of $j$ before the $(k+1)$-th occurrence of $i$, and $g-k$ occurrences of $j$ 
after that occurrence.
\end{proof}

\begin{corollary}
\label{c:interlacing}
In any interesting walk,
there is exactly one occurrence of node $i\in Z_0$ between every two consecutive occurrences of $j\notin Z_0$,
and no occurrences of $i\in Z_0$ neither before the first nor after the last occurrence of~$j\notin Z_0$.  
\end{corollary}

\begin{lemma}
\label{l:purpleyellow}
Any interesting walk
$W_0$ can be represented as 
\begin{equation}
\label{e:purpleyellow}
W_0=P Q P_1 V,
\end{equation}
where $P$ and $P_1$ contain all nodes not in $Z_0$ exactly once and only them, 
$Q$ contains all nodes of $Z_0$ exactly once and only them,
and $V$ is a walk starting with a node in $Z_0$.\\
\end{lemma}
\begin{proof}
Define $P$ such that $W_0= P V_0$, where all nodes of
$P$ are not in $Z_0$ and the first node of $V_0$ is in $Z_0$. By Lemma~\ref{l:interlace} part (ii),
$P$ contains all nodes not in $Z_0$ exactly once (and only them).

Define $Q$ as the subpath of $V_0$ such that $V_0= Q U_1$, where all nodes of
$Q$ are in $Z_0$ and the first node of $U_1$ is not in $Z_0$. That node also occurs once in $P$.
 By Lemma~\ref{l:interlace} part (i), all nodes of $Z_0$ should occur between the two occurrences of that node
exactly once, and therefore $Q$ contains all such nodes exactly once (and only them). 

Define $P_1$ such that $U_1= P_1V$, where all nodes of
$P_1$ are not in $Z_0$ and the first node of $V$ is in $Z_0$.  That node also occurs once in $Q$. 
 By Lemma~\ref{l:interlace} part (ii), all nodes that are not in $Z_0$ should occur between the two occurrences of 
that node exactly once, and therefore 
$P_1$ contains all such nodes exactly once (and only them).
\end{proof}

\begin{proof}[{\bf Proof of Proposition~\ref{p:W0}}]
Let~$W_0$ be an interesting walk and $P$~and~$Q$ defined as in~\eqref{e:purpleyellow}.
We want to show

\begin{equation}
\label{e:W01}
W_0= P (Q P)^g.
\end{equation}

We start with the decomposition~\eqref{e:purpleyellow}. Define 
$k\in Z_0$ as the last node of $Q$ and $Q'$ by~$Q=Q'k$. 

\begin{figure}
\centering
\begin{tikzpicture}[>=latex']
\node[circle,very thick,draw,minimum size=9mm] (n1) at (-6,0) {};
\node[circle,very thick,draw,minimum size=9mm] (nj) at (-2,0) {};
\node[circle,very thick,draw,minimum size=9mm,red] (ni) at (2,0) {$k$};
\node[circle,very thick,draw,minimum size=9mm] (nf) at (6,0) {};

\draw[->,very thick] (n1) -- (nj) node [midway, above] {$P$};
\draw[->,very thick,red] (nj) -- (ni) node [midway, above] {$Q'$};
\draw[->,very thick] (ni) -- (nf) node [midway, above] {$U$};

\draw[->,very thick] (ni) .. controls +(145:4) and +(105:4) .. (ni) node [pos=0.2, left] {$P_1$} node [pos=0.8, right] {$V_1$};
\draw[->,very thick] (ni) .. controls +(75:4) and +(35:4) .. (ni) node [pos=0.5, above] {$W_\alpha$, $\alpha\geq 2$};
\end{tikzpicture}
\caption{Structure of interesting walk $W_0$ as of \eqref{e:frompurpleyellow}}
\label{f:frompurpleyellow}
\end{figure}

As a first step, using~\eqref{e:purpleyellow} and observing $g$ occurrences of $k$ we can immediately obtain
\begin{equation}
\label{e:frompurpleyellow}
W_0=PQ' kW_1kW_2\dots k W_{g-1}kU,
\end{equation}
see Figure~\ref{f:frompurpleyellow}.
Decomposition~\eqref{e:purpleyellow} also implies that the subpath $PQ$ in the beginning should be followed
by a sequence $P_1$ containing all non-critical nodes once (and only them), in any interesting walk. Therefore we have
$W_1=P_1 V_1$ for some $V_1$.

\begin{figure}
\centering
\begin{tikzpicture}[>=latex']
\node (n1) at (-6,0) {};
\node[circle,very thick,draw,minimum size=9mm] (nkp) at (-2,0) {$k'$};
\node[circle,very thick,draw,minimum size=9mm,red] (nk) at (2,0) {$k$};
\node (nf) at (6,0) {};
\node[circle,very thick,draw,minimum size=9mm] (nkp2) at (4,2) {$k'$};

\draw[->,very thick] (n1) -- (nkp);
\draw[->,very thick] (nkp) -- (nk);
\draw[->,very thick] (nk) -- (nf);

\draw[->,very thick] (nk) .. controls +(0.8,1.2) .. (nkp2); 
\draw[->,very thick] (nkp2) .. controls +(-0.8,-1.2) .. (nk);
\draw[->,very thick] (nkp2) .. controls +(180:3) and +(140:3) .. (nkp2) node [pos=0.5, left] {$R$};
\draw[->,very thick,dashed] (nkp) .. controls +(110:3) and +(70:3) .. (nkp) node [pos=0.5, above] {$R$};

\draw[|->,very thick] (1.2,2.7) -- (-1.6,2.5) node [midway, above] {move $R$};
\end{tikzpicture}
\caption{Node $k'$ appears twice in $C_1=kPV_1k$}
\label{f:moveR}
\end{figure}

The cycle $C_1=kP_1V_1k$ contains each node (critical and non-critical) no more than once.
Indeed, if a node $k'$ occurred in 
$C_1$ twice then we would decompose $C_1=k\dots k'Rk'\dots k$,
replace the node $k'$ in $PQ$ by $k'Rk'$ and delete $Rk'$ from $C_1$, see Figure~\ref{f:moveR}. This 
would result in a new walk with $g$ consecutive closed walks some of which can be 
removed, 
resulting in a walk of a smaller length and showing that the initial walk was not interesting,
a contradiction.

\begin{figure}
\centering
\begin{tikzpicture}[>=latex']
\node (n1) at (0,0) {};
\node[circle,very thick,draw,minimum size=9mm] (nj1) at (2.4,0) {$j$};
\node[circle,very thick,draw,minimum size=9mm,red] (nk1) at (4.8,0) {$k$};
\node[circle,very thick,draw,minimum size=9mm] (nj2) at (7.2,0) {$j$};
\node[circle,very thick,draw,minimum size=9mm,red] (nk2) at (9.6,0) {$k$};
\node (nf) at (12,0) {};

\node at (2.4,1) {$PQ'$};
\node at (7.2,1) {$P_1V_1$};

\draw[->,very thick] (n1) -- (nj1) node [midway, below] {$S$};
\draw[->,very thick] (nj1) -- (nk1) node [midway, below] {$T$};
\draw[->,very thick] (nk1) -- (nj2) node [midway, below] {$S'$};
\draw[->,very thick] (nj2) -- (nk2) node [midway, below] {$T'$};
\draw[->,very thick] (nk2) -- (nf);

\draw[<->,very thick,dashed] (3.6,-0.7) .. controls +(2.4,-1) .. (8.4,-0.7) node [midway,below] {exchange};

\end{tikzpicture}
\caption{Exchange if $PQ' \neq P_1V_1$}
\label{f:exchange}
\end{figure}

Since any node is contained in $C_1$ no more than once,
$V_1$ consists of nodes of $Z_0$ only. 
Comparing~\eqref{e:W01} with~\eqref{e:frompurpleyellow} 
we need to prove that $V_1=Q'$ and that $P_1=P$. 
For that, take any node $j\in P_1V_1$. It also occurs in $PQ'$ since that path contains all nodes but $k$.
Consider the following decompositions $PQ'k=SjTk$ and $P_1V_1k=S'jT'k$. If we assume that $Q'\neq V_1$ or 
$P_1\neq P$ then for some $j$ the sets of nodes of $S$ and $S'$ differ
or the sets of nodes of $T$ and $T'$ differ. Assume the latter (the case of different $S$ and $S'$ is treated similarly).
By replacing $PQ'k$ with $SjT'k$ and 
$P_1V_1k$ by $S'jTk$ as in Figure~\ref{f:exchange} (in other words, by exchanging $T$ and $T'$) 
we obtain a new interesting walk. We now prove that 
it is not of the form~\eqref{e:purpleyellow}, in contradiction with Lemma~\ref{l:purpleyellow}.

Indeed, we have $SjT'k=\Tilde{P}\Tilde{Q}$ where $\Tilde{P}$ consists only of nodes not in $Z_0$, and $\Tilde{Q}$ consists only of nodes
in $Z_0$. Similarly, $S'jTk=\Tilde{P_1}\Tilde{Q_1}$ where $\Tilde{P_1}$ consists only of nodes not in $Z_0$, and $\Tilde{Q_1}$ consists only of nodes
in $Z_0$. However, the set of nodes of $Tk$ is a complement of the set of nodes of $Sj$ 
(recall that $PQ'k=SjTk$ and all nodes occur in $PQ$ exactly once)
and that of $T'k$ is not (since $T$ and $T'$ have different node sets), and this implies that $\Tilde{P}$ or $\Tilde{Q}$ miss some nodes in contradiction with 
Lemma~\ref{l:purpleyellow}. Hence $P_1=P$ and $V_1=Q'$.

Generalizing the cycle~$C_1$, define $C_\alpha = kW_\alpha k$ for all $1\leq\alpha\leq g-1$.
Since we can exchange any two $C_\alpha$ without changing neither the length, nor the weight of the walk, the decomposition of $C_1$
is also true for any~$C_\alpha$, that is $C_{\alpha}=kPQ'k$ for all $\alpha$.

Now, each critical node occurs in the walk $PQ' kW_1kW_2\dots k W_{g-1}k =(PQ)^g$ exactly $g$ times,
hence, by Lemma~\ref{l:interlace} part (ii), $U$ contains all non-critical nodes exactly once, and only them.
So we have obtained that $W_0=(PQ)^gU$ where $U$ contains all non-critical nodes exactly once, and only them.

It remains to show that $V=P$.
In  $\digr(\trans{A})$ (the graph of the transpose of $A$) there  is also an interesting walk.
Since $(\trans{A})^m=\trans{(A^m)}$ for all $m\geq 1$, Walk $\trans{W_0}$, that starts with the end of $W_0$
and goes to the beginning of $W_0$ via exactly the same nodes listed in the opposite order is an interesting walk on~$\digr(\trans{A})$.
On one hand, by construction $\trans{W_0}=\trans{V}(\trans{Q}\trans{P})^{g}$, where $\trans{V}$, $\trans{Q}$ and $\trans{P}$ contain 
the same nodes as $V$, $Q$ and $P$ listed in the opposite order. On the other hand, applying the argument above we get a decomposition 
of the form $\trans{W_0}=(\overline{P}\;\overline{Q})^g \overline{V}$, and, since $g\ge 2$, we conclude that $\trans{V}=\trans{P}=\overline{P}$.
This implies $V=P$, so the decomposition~\eqref{e:W01} is established.

In order to obtain~\eqref{e:W0} we renumber the nodes of $\digr(A)$ in such a way that $QP=1\dots n$.
\end{proof}

\begin{corollary}
\label{c:coprime}
If $T_1(A)=\DM(g,n)$ then $n$ and $g$ are coprime (Condition~\ref{Cgnprime}).
\end{corollary}
\begin{proof}
\if{
If $n$ and $g$ are not coprime, then their least common multiple denoted by $p$ is 
smaller than $gn$. We have $gn=tp$ where $t>1$. Writing $gn=p+(t-1)p$ we observe that
the Hamiltonian cycles forming a closed walk of length $p$ can be removed and the rest of the Hamiltonian cycles can be kept.
This leads to a shorter walk with larger weight, in contradiction with the double optimality
of $W_0$.
}\fi

If $n$ and $g$ are not coprime, then $d=\gcd(n,g)>1$. We have $g=pd$ and $n=qd$ for some $p$ and $q$. Let 
$W_0$ be given by~\eqref{e:W0} and $W_1$ be $(g+1)\dots n (1\dots n)^{g-p}$. Since $pn=gq$, we have
$l(W_1)$ and $l(W_0)$ are congruent modulo $g$, and since $p(W_1)\geq p(W_0)$, we obtain that 
$W_0$ is not twice optimal, so $T_1(A,B)<\DM(g,n)$ by Proposition~\ref{p:W0}, a contradiction.
\end{proof}

\subsection{Proof of Necessity}

In this subsection, we finish the proof of necessity of Conditions 1.--6.
and the last statement of the theorem. We assume~$T_1(A)=\DM(g,n)$.

Condition~\ref{Cgnprime} was proved as Corollary~\ref{c:coprime}.
By Proposition~\ref{p:critgraph}, $\crit(A)$ is strongly connected and contains only one cycle with length~$g$ denoted by~$Z_0$. (Condition~\ref{CGc}).

We now turn to the proof of Conditions~\ref{CZO}~\ref{CA2} and~\ref{CB1}, which will be proved together.
The core of the proof is split into Lemmas~\ref{l:propertyA}, \ref{l:propertyB} and~\ref{l:CSRineq} below.

By Proposition~\ref{p:W0} there is a unique twice optimal walk $W_0$ of length $\DM(g,n)+g-1$. 
After renumbering the nodes we can assume that $(i,i+1)$ for $1\leq i\leq (n-1)$
and $(n,1)$ are the arcs of a Hamiltonian cycle of $\digr(A)$, and that nodes $\{1,\ldots,g\}$, are the nodes of~$Z_0$.

Notice that we have not yet proved $Z_0=1\dots g1$ (condition~\ref{CZO}) since we do not know the arcs of~$Z_0$.
 
Any occurrence of a node in $W_0$ can be encoded by its {\em position} in that walk.
We now define what we mean by position. We assume that the first occurrence of node $n$ has position~$0$, and the position
of any node $i$ in the $k$th copy of $1\dots n$ (called {\em period}) is $i+(k-1)n$, and the position of the $i$th node in the part of the walk before the first occurrence of $1$
(for $i\in\{g+1,\ldots,n\}$) is $i-n$. Note that
these positions 
(and only these) are non-positive. 

We will be interested in the set of subwalks of $W_0$ from $i$ to $j$, with length $1$ modulo $g$.
Denote this set by $\walkslen{i}{j}{1,g}[W_0]$. 
This set is nonempty if and only if there is an occurrence of node $i$ 
at 
some position denoted by $N_i^b$ and there is an occurrence of node $j$ at some position denoted by $N_j^e$ such that 
$N_j^e>N_i^b$ and $N_j^e-N_i^b\equiv_g 1$.

Consider the following properties of subwalks of $W_0$:

{\bf Property A:}  We say that a subwalk $W\in\walkslen{i}{j}{1,g}[W_0]$ has this property
if it goes through one of the first $g$ nodes (i.e. a node of~$Z_0$). 

{\bf Property B:}  We say that a subwalk $W\in\walkslen{i}{j}{1,g}[W_0]$ has this property if 
after replacing $W$ in $W_0$ by the arc $(i,j)$ where $i$ and $j$ are the beginning node and the
end node of $W$ respectively, the resulting walk $W'_0$ goes through one of the first 
$g$ nodes.   


A subwalk $W\in\walkslen{i}{j}{1,g}[W_0]$ does not have Property B if and only if it begins in the tail, ends in the last period of $W_0$
and has $i,j>g$.

Define $A_1$, $B_1$ and $A_2$ by~\eqref{e:A1},~\eqref{e:B1} and~\eqref{e:A2}.


\begin{lemma}
\label{l:propertyA}
If $W\in\walkslen{i}{j}{1,g}[W_0]$ has Property A then $p(W)\leq (CSR)_{i,j}[A]$.
\end{lemma}
\begin{proof}
Property A means that $W$ belongs to 
$\walkslennode{i}{j}{1,g}{Z_0}$ on $\digr(A)$, since the first $g$ nodes are the node set of $Z_0$.
As $(CSR)_{i,j}[A]$ is the largest weight of walks in $\walkslennode{i}{j}{1,g}{Z_0}$ on $\digr(A)$ (recall $\lambda(A)=\1$
and Proposition~\ref{p:CSRprops} part~(\ref{CSR-walks})), the claim 
follows.
\end{proof}

\begin{lemma}
\label{l:propertyB}
If  $W\in\walkslen{i}{j}{1,g}[W_0]$ has Property B and $j\neq i+1$ and $(i,j)\neq (n,1)$ then $a_{i,j}<p(W)$.
\end{lemma}
\begin{proof}
If $a_{i,j}\geq p(W)$ then replacing $W$ by $(i,j)$ in $W_0$ we get a walk with the same length
modulo $g$ as $W_0$, whose weight is not less
than $p(W_0)$ and whose length is strictly less than $\ell(W_0)$. This contradicts the fact that 
$W_0$ is twice optimal. 
\end{proof}

\begin{lemma}
\label{l:CSRineq}
\begin{itemize}
\item[1.] If $j\not\equiv_g(i+1)$ then $a_{i,j}<(CSR)_{i,j}[A_1]$ and 
$a_{i,j}<(CSR)_{i,j}[A_1\oplus B_1]$;
\item[2.] If $j\equiv_g(i+1)$ and $i$ or $j$ belong to $\{1,\ldots,g\}$ but $j\neq i+1$ and
$(i,j)\neq (g,1)$, then $a_{i,j}<(CSR)_{i,j}[A_1]$ and $a_{i,j}<(CSR)_{i,j}[A_1\oplus B_1]$;
\item[3.] The arcs $(i,i+1)$ for $1\leq i\leq g-1$ and $(g,1)$ are critical.
 \end{itemize}
\end{lemma}
\begin{proof}
 We will examine the existence of walks with Properties A and B in the cases of our interest,
 and apply Lemmas~\ref{l:propertyA} and~\ref{l:propertyB} .

\begin{figure}
\centering
\begin{tikzpicture}[>=latex']
\node at (1,1) {Case 1};

\node (n1) at (0,0) {};
\node[circle,very thick,draw,minimum size=9mm,red] (nj1) at (2.4,0) {$1$};
\node[circle,very thick,draw,minimum size=9mm] (nk1) at (4.8,0) {$i$};
\node[circle,very thick,draw,minimum size=9mm,red] (nj2) at (7.2,0) {$1$};
\node[circle,very thick,draw,minimum size=9mm] (nk2) at (9.6,0) {$j$};
\node (nf) at (12,0) {};

\node at (2.4,-0.7) {$\in W_0'$};
\node at (4.8,-0.7) {$N^b_i$};
\node at (7.2,-0.7) {$\in W$};
\node at (9.6,-0.7) {$N^e_j$};

\draw[->,very thick] (n1) -- (nj1);
\draw[->,very thick] (nj1) -- (nk1);
\draw[->,very thick] (nk1) -- (nj2);
\draw[->,very thick] (nj2) -- (nk2);
\draw[->,very thick] (nk2) -- (nf);

\draw [very thick,decorate,decoration={brace,amplitude=10pt}] (4.8,0.7) --
(9.6,0.7) node [midway,above=0.4cm] {$W$};
\end{tikzpicture}
\begin{tikzpicture}[>=latex']
\node at (1,1) {Case 2.1};

\node (n1) at (0,0) {};
\node[circle,very thick,draw,minimum size=9mm,red] (nj1) at (2.4,0) {$1$};
\node[circle,very thick,draw,minimum size=9mm,red] (nk1) at (4.8,0) {$i$};
\node[circle,very thick,draw,minimum size=9mm] (nj2) at (7.2,0) {$j$};
\node[circle,very thick,draw,minimum size=9mm,red] (nk2) at (9.6,0) {$1$};
\node (nf) at (12,0) {};

\node at (4.8,-0.7) {$N^b_i$};
\node at (4.8,-1.2) {$\in W_0' \cap W$};
\node at (7.2,-0.7) {$N^e_j$};

\draw[->,very thick] (n1) -- (nj1);
\draw[->,very thick,red] (nj1) -- (nk1);
\draw[->,very thick] (nk1) -- (nj2);
\draw[->,very thick] (nj2) -- (nk2);
\draw[->,very thick] (nk2) -- (nf);

\draw [very thick,decorate,decoration={brace,amplitude=10pt}] (4.8,0.7) --
(7.2,0.7) node [midway,above=0.4cm] {$W$};
\end{tikzpicture}
\begin{tikzpicture}[>=latex']
\node at (1,1) {Case 2.2};

\node (n1) at (0,0) {};
\node[circle,very thick,draw,minimum size=9mm] (nj1) at (2.4,0) {$i$};
\node[circle,very thick,draw,minimum size=9mm,red] (nk1) at (4.8,0) {$1$};
\node[circle,very thick,draw,minimum size=9mm,red] (nj2) at (7.2,0) {$1$};
\node[circle,very thick,draw,minimum size=9mm,red] (nk2) at (9.6,0) {$j$};
\node[circle,very thick,draw,minimum size=9mm] (nf) at (12,0) {$n$};

\node at (2.4,-0.7) {$N^b_i$};
\node at (9.6,-1.2) {$\in W_0' \cap W$};
\node at (9.6,-0.7) {$N^e_j$};

\draw[->,very thick] (n1) -- (nj1);
\draw[->,very thick] (nj1) -- (nk1);
\draw[->,very thick] (nk1) -- (nj2);
\draw[->,very thick,red] (nj2) -- (nk2);
\draw[->,very thick] (nk2) -- (nf);

\draw [very thick,decorate,decoration={brace,amplitude=10pt}] (2.4,0.7) --
(9.6,0.7) node [midway,above=0.4cm] {$W$};
\end{tikzpicture}
\caption{Cases 1, 2.1, and 2.2}
\label{f:necessity}
\end{figure}

1:  Examine the case $j\not\equiv_g (i+1)$ . We take the occurrence of $i$ in the first period, that is, $N_i^b=i$. Then, since  $j\not\equiv_g (i+1)$, the (unique) occurrence of $j$ for which $N_j^e-i\equiv_g 1$ and $N_j^e>i$ exists in some other period.
They have both Property~A 
and Property~B, and  Lemmas~\ref{l:propertyA} and~\ref{l:propertyB} imply that 
$a_{i,j}<(CSR)_{i,j}[A]$. See Figure~\ref{f:necessity} for this case as well as cases 2.1 and 2.2 
described below.

2: We need to examine the following two cases: 2.1 $j\equiv_g (i+1)$, $j>i$ and $i\leq g$ and
2.2 $j\equiv_g (i+1)$, $j<i$, $i>g$, $j\leq g$.

{\em Case 2.1.}  We can take the occurrences of $i$ and $j$ in any period $k$ at positions 
$N_i^b=i+(k-1)n$ and $N_j^e=j+(k-1)n$. 

{\em Case 2.2.}  Take the occurrence of $i$ with $N_i^b=i-n$ and the occurrence of $j$ with 
$N_j^e=j+(g-1)n$ (in the last period). 

In both cases, the walks defined by these occurrences have both Property~A and
Property~B. 
By  Lemmas~\ref{l:propertyA} and~\ref{l:propertyB}, we have $a_{i,j}<(CSR)_{i,j}[A]$ if 
$j\neq i+1$.  Note that here and in 1. above we still have to argue that $A$ can be replaced with $A_1$ and $A_1\oplus B_1$. 

3:  If an arc $(i,j)$ is critical then $a_{i,j}=(CSR)_{i,j}[A]$. Hence if $a_{i,j}<(CSR)_{i,j}[A]$ then $(i,j)$ is non-critical.
As $\{1,\ldots,g\}$ are nodes of a critical cycle of length $g$ and parts 1. and 2. above imply that 
 $(i,i+1)$ for $1\leq i\leq g-1$ and $(g,1)$ 
are the only arcs between the first $g$ nodes that
can have $a_{i,j}=(CSR)_{i,j}[A]$, so these arcs are critical.   
 
Thus $1\dots g1$ is a critical cycle of~$A$ and $\lambda(A_1)=\lambda(A_1\oplus B_1)=\1$, and therefore 
$A$ can be replaced with $A_1$ and $A_1\oplus B_1$ 
first in statement and proof of Lemma~\ref{l:propertyA} and therefore also in the proofs of 1. and 2. above.
%
\end{proof}

Condition~\ref{CZO} of the theorem follows now from 
Lemma~\ref{l:CSRineq}~$3.$,
Condition~\ref{CA2} follows from 
Lemma~\ref{l:CSRineq}~$2.$, and Condition~\ref{CB1} is implied by the following Lemma.

\begin{lemma}
\label{l:A1}
If $j\equiv_g(i+1)$ and $j>i>g$ then $a_{i,j}<(A_1)_{i,j}^{j-i}$.
\end{lemma}
\begin{proof}
We can take the occurrences of $i$ and $j$ 
at positions $i-n$ and $j-n$.  
The resulting walk has 
Property~B. Hence by Lemma~\ref{l:propertyB} we have $a_{i,j}<p(W)$ for $j\neq i+1$. As $W$ is also a unique (and hence optimal) walk on 
$\digr(A_1)$ from $i$ to $j$ and having length $j-i$ we obtain that $p(W)=(A_1)_{i,j}^{j-i}$.
\end{proof}

It remains to obtain Condition~\ref{CB2}. Since $T_1(A)=\DM(g,n)$, we have 
 $(A^{\DM(g,n)-1})_{g+1,n}<(CS^{\DM(g,n)-1}R)_{g+1,n}[A]$,\\ and hence
$(B_1^{\DM(g,n)-1})_{g+1,n}< (CS^{\DM(g,n)-1}R)_{g+1,n}[A]$.

Let us argue that $CS^tR[A]=CS^tR[A_1]$ for all $t$. Indeed, the equality
$CS^tR[A]=CS^tR[A_1\oplus B_1]$ follows from Lemma~\ref{l:perturbation}
since $A_2<CSR[A_1\oplus B_1]$ by Lemma~\ref{l:CSRineq},
and the equality $CS^tR[A_1\oplus B_1]=CS^tR[A_1]$, follows from 
Lemma~\ref{l:CSRA=CSRAB}. We also have $CS^{\DM(g,n)-1}R[A_1]=CS^{n-1}R[A_1]$ since
$\lambda(A_1)=\1$ and $\DM(g,n)\equiv_g n$.

The remaining uniqueness statements have been proved in Remark~\ref{r:W0}.

\section{Matrices Attaining the Wielandt Bound}\label{sec:wiel}

This section is devoted to the proof of Theorem~\ref{t:wiel}.

If $\wiel(n)>\DM(g,n)$, then the Wielandt bound cannot be attained. Hence we are only 
interested in the case when $\wiel(n)\leq\DM(g,n)$.
Observe that $\wiel(n)=\DM(n-1,n)$, and therefore 

$$\DM(g,n)\geq\wiel(n)\Leftrightarrow \DM(g,n)\geq \DM(n-1,n)\Leftrightarrow g\geq n-1$$ 
for any $n\geq 2$. 

Thus we have two cases: $g=n-1$ and $g=n$.

{\em In case $g=n-1$}, we have $\DM(n-1,n)=\wiel(n)$, and observe that
Conditions~\ref{Cgnprime}, \ref{CB1} and~\ref{CB2} of Theorem~\ref{t:mainres}
trivially hold, in view of Remark~\ref{r:n<2g}. Next, both sufficiency and necessity as well as the
last part of the statement, for $g=n-1$,
follow as a special case of the corresponding claims in Theorem~\ref{t:mainres}. 

{\em In case $g=n$}, let us prove the sufficiency and necessity of Condition~2.

{\em Sufficiency:} By Lemma~\ref{l:perturbation}, Condition~2 implies that
$T_1(A)=T_1(A_1)$, so it suffices
to show that $T_1(A_1)=\wiel(n)$. For this, note that for $\tilde{A}_1$ with entries defined by
\begin{equation}
\label{e:tildeA}
(\tilde{A}_1)_{i,j}=
\begin{cases}
\bunity, &\text{if $(A_1)_{i,j}\neq \0$},\\
\bzero, &\text{otherwise.}
\end{cases}
\end{equation}
(Wielandt's example, see e.g.~\cite{BR}) 
we have $T(\tilde{A}_1)=\wiel(n)$, meaning that there exist $i,j$ such that 
$(\tilde{A}_1)^{\wiel(n)-1}_{i,j}=\0$. This implies that also 
\mbox{$(A_1)^{\wiel(n)-1}_{i,j}=\0$} and hence $T_1(A_1)\geq \wiel(n)$ (recalling that
$T_1(A_1)=T(A_1)$). However we also have $T_1(A_1)\leq\wiel(n)$ and hence 
$T_1(A_1)=\wiel(n)$.

\medskip
  
{\em Necessity:}   
To prove necessity we will need the following result, which is analogous to
Proposition~\ref{p:W0}. Since the conditions are invariant under scalar multiplication,
we will assume $\lambda(A)=\1$ in the rest of this section.

\begin{proposition}
\label{p:wielwalk}
Let $A\in\Rpnn$ for $n\geq 1$  and $g=n$. 
If $A$ has $T_1(A)=\wiel(n)$, then
there exists a unique twice optimal walk of 
length $\wiel(n)+n-1$. It is of the form
\begin{equation}
\label{e:wielwalk}
W_0=n(1 \dots (n-1))^{n-1} 1\dots n,
\end{equation}
where, after appropriate renumbering,  $(i,i+1)$ for $1\leq i< n$ and $(n,1)$
are the arcs of the unique critical cycle.
\end{proposition}

\begin{proof}
This is an improvement on the proof of Theorem~\ref{t:bounds} in~\cite{MNS} with extra care on the extremal cases.

Let $Z_0$ be a critical cycle, which is of length~$n$.
This is the only critical cycle 
(up to choice of its first node) 
as any extra arc would lead to a shorter cycle.

Let $W$ be a twice optimal walk.
Denote its first node by~$i$ and its last node by~$j$.
Let $i_{\alpha}j_{\alpha}$, for $\alpha=1,\ldots,m$ be the noncritical 
arcs of $W$. 
For every such arc of $W$, insert a copy of $Z_0$ as 
a walk beginning and ending at $j_{\alpha}$. Denote the resulting walk by $W'$. 
Observe that for each noncritical arc $i_{\alpha}j_{\alpha}$ we can detect the
node $i_{\alpha}$ occurring in the subsequent copy of $Z_0$. This gives rise to a cycle $C_{\alpha}$ consisting
of a shortcut $i_{\alpha}j_{\alpha}$ and a number of critical arcs. Thus we obtain the following
decomposition in term of the multiset of its arcs:
$$
M(W')=M(P) \cup \bigcup_{\alpha=1}^m M(C_{\alpha}) \cup \bigcup_{\beta=1}^k M(Z_\beta) \enspace,
$$
where $M(V)$ denotes the multiset of arcs of a walk~$V$, $m$ is the number of noncritical arcs in the original walk, the~$Z_\beta$ are critical cycles, and $P$ is a critical path from $i$ to $j$.
Since~$Z_0$ is the only critical cycle, we have $Z_\beta = Z_0$ for all~$\beta$.
We can remove $k-1$ copies of  $Z_0$ and get an optimal walk in $\walkslen{i}{j}{1,n}$ of smaller length.  
Denote the resulting walk by $W''$.

Further, since $W''$ has the largest weight in $\walkslen{i}{j}{1,n}$, one cannot remove cycles from it maintaining the length modulo $n$,
and hence $m\leq n-1$ by Lemma~\ref{l:HA}. We distinguish two cases:

1) $M(P)\cup \bigcup_{\alpha=1}^m M(C_{\alpha})$ is connected, in which case also the $k$\textsuperscript{th} copy of $Z_0$ 
can be removed. The length of the resulting walk is bounded by $(n-1)+(n-1)^2< (n-1)^2+n$. 
Thus, $T_1(A) < \wiel(n)$, a contradiction, which shows that this case is impossible.

2) $M(P)\cup \bigcup_{\alpha=1}^m M(C_{\alpha})$ is disconnected. 
Then we cannot remove $Z_0$ from~$W''$. However, in this
case there is a cycle $\Hat{C}$ such that $l(\Hat{C})+l(P)\leq n-1$. Therefore the length of $W''$ is bounded by 
$n+(n-1)+(n-2)(n-1)=(n-1)^2+n$. 

Further, in case~2, the length of all cycles~$C_\alpha$ is bounded by $n-2$, unless one of the connected
components of $M(P)\cup\bigcup_{\alpha=1}^m M(C_{\alpha})$ has size $1$. In the former case the length is bounded by 
$n+(n-1)+(n-2)^2<(n-1)^2+n$, which is again impossible. It thus remains to treat two
subcases:\\
2a) There is a loop (or possibly, several copies of the same loop)
disconnected from $P$, and the rest of the cycles of $M(P)\cup \bigcup_{\alpha=1}^m 
M(C_{\alpha})$ connected to $P$;\\
2b) $l(P)=0$ and there are $n-1$ cycles of length at most $n-1$ disconnected from $P$.

In  subcase 2a, the length of all cycles is bounded by $n-2$, since any cycle of length $n-1$
could be combined with the loop and removed. In this case, the length of the walk is again bounded by 
$n+(n-1)+(n-2)^2<(n-1)^2+n$, which is impossible.  

In subcase 2b, we have $i=j$.  Length $l(W'')$ can reach the length $(n-1)^2+n$ only if all cycles $C_{\alpha}$, not
containing $i=j$, are of length $n-1$. However, by construction,
every such cycle should contain just one noncritical arc, and this has to be a $1$-shortcut
bypassing $i=j$. However, there is only one $1$-shortcut bypassing $i=j$
which implies that all these cycles are identical and $W''$ has to be of the form~\eqref{e:wielwalk}, after renumbering the nodes in such a way 
that $Z_0=1\ldots n1$ and $i=n$.
\end{proof}

As in the proof of Theorem~\ref{t:mainres}, any occurrence of a node in $W_0$
of Proposition~\ref{p:wielwalk} can be encoded by its {\em position} in that
walk.
Node $n$ occurs there only twice. Its first occurrence has position 
$0$, and its second occurrence has position $(n-1)^2+n$. 
The $k$th occurrence of node $i$, for $1\leq i\leq n-1$ and $k=1,\ldots,n$ has position 
$i+(k-1)(n-1)$. 

A subwalk of length $1$ modulo $n$ from an occurrence of node $i$ at position $N^b_i$ to an occurrence of node  
$j$ at position $N^e_j$ exists if and only if $N^b_i<N^e_j$ and $N^e_j-N^b_i\equiv_n 1$.

We will need the following observation:
\begin{lemma}
\label{l:CSRA1}
If there are occurrences $N^b_i$ of node $i$ and $N^e_j$ of node $j$ such that the subwalk from 
$N^b_i$ to $N^e_j$
has length $1$ modulo $n$  and $N^e_j-N^b_i>1$, then $a_{i,j}<(CSR[A_1])_{i,j}$.
\end{lemma}
\begin{proof}
Denote by $W$ the subwalk from $N^b_i$ to $N^e_j$.   This subwalk uses only the arcs of the digraph~$\digr(A_1)$,
where it is optimal among the walks with the same length modulo $n$, and goes through critical nodes (since all nodes 
are critical). Hence $p(W)=(CSR[A_1])_{i,j}$.  
If $W$ is replaced with $(i,j)$ then the resulting 
walk also goes through the critical nodes, and therefore we must have $a_{i,j}<p(W)$, for otherwise
$W_0$ is not twice optimal. Hence the claim. 
\end{proof}

Using Lemma~\ref{l:CSRA1}, for the necessity of Condition~2, it suffices to prove that the subwalks 
with length $1$ modulo $n$ but larger than $1$ exist for any $(i,j)$ except possibly for $(i,j)=(n-1,1)$, $(i,j)=(n,1)$, or 
$j=i+1$ and $i\in\{1,\ldots,n-1\}$, which are the arcs of $\digr(A_1)$.  

{\bf Case 1.} $i=j=n$. We have two occurrences of $n$ with $N^b_n=0$ and $N^e_n=(n-1)^2+n$.
As $(n-1)^2+n\equiv_n 1$, the walk with required properties exists. 

{\bf Case 2.} $i=n$ and $j\neq n$. In this case $N^b_n=0$ and we can choose $N^e_j=j+(k-1)(n-1)$ 
with $N^e_j\equiv_n 1$, since $k\in\{1,\ldots n\}$ and $n-1$ and $n$ are coprime.

{\bf Case 3.} $i\neq n$ and $j=n$. This case is symmetrical to case 2.

{\bf Case 4.} $i\neq n$, $j\neq n$, $(i,j)\neq (n-1,1)$, and $j\neq i+1$. 
Observe that $j\neq i+1$ is equivalent to $j\not\equiv_n (i+1)$ because $1\leq i \leq j\leq n$. \\
Take $N_b^i=i$ and $N_e^j=j+(k-1)(n-1)$ where $k\in\{1,\ldots,n\}$. 
Then $N_e^j-N_b^i=(j-i)+(k-1)(n-1)$. As $n$ and $n-1$ are coprime, there is a $k\in\{1,\ldots,n\}$
such that $(j-i)+(k-1)(n-1)\equiv_n 1$, and since $j\not\equiv_n i+1$ we have $k>1$ and hence 
$N_e^j-N_b^i> 1$. The case $N_e^j-N_b^i=1$ is only possible if $i=(n-1)$ and $j=1$, the case which was 
excluded.

\medskip

The uniqueness of the renumbering follows from the uniqueness of $W_0$.

\section{Matrices with Critical Columns Attaining the Bounds}\label{sec:crit}

This section is devoted to the proof of Theorems~\ref{t:CritColDM} and~\ref{t:CritColWiel}. As before, we assume that~$\lambda(A)=\1$.

Let us first notice that the transient of critical rows and columns is at most~$T_1(A)$ because if $i$ or $j$ is critical then $(CS^tR\oplus
B_N^t)_{i,j}=(CS^tR)_{i,j}$. Thus, if the transient of a critical row or column reaches the bound, so does~$T_1(A)$ and $A$~ belongs to the class defined by Theorem~\ref{t:mainres} (except if $\g(\crit(A))=1$) or~\ref{t:wiel}. 

\begin{proof}[Proof of Theorem~\ref{t:CritColDM}]
Let~$A$ be a matrix with~$\g(\crit(A))=g$ and a critical row~$i_0$ whose transient is $\DM(g,n)$ and show that the index of~$\crit(A)$ is~$\DM(g,n)$.
Then, the same is true for a critical column by transposition of the matrix and the converse of the theorem follows from~\cite{MNSS}[Lemma 8.2].
We also have $T_1(A)=\DM(g,n)$, by the above argument.

We assume without loss of generality that $A$~is strictly visualized.

By Proposition~\ref{p:critgraph}, $\crit(A)$ is strongly connected and contains a unique (up to choice of the starting node) cycle of length~$g$ denoted by~$Z_0$.

Since~$i_0$ is critical, the transient of row~$i_0$ reaches the bound means that there is a $j_0$ such
that $A_{i_0j_0}^{\DM(g,n)-1}<(CS^{n-1}R)_{i_0j_0}$ and by Proposition~\ref{p:twiceopt}, there is an interesting walk~$W_0$ from~$i_0$ to~$j_0$.

{\em If $g\ge 2$}, then $A$ satisfies the conditions of Theorem~\ref{t:mainres}. We assume without loss of generality that~$A$ has been
reordered as in the theorem.

By Proposition~\ref{p:W0}, we have~$(i_0,j_0)=(g+1,n)$, so that $g+1$ is critical.

It remains to understand~$\crit(A)$.
By Lemma~\ref{l:perturbation}, it is contained in~$\crit(A_1\oplus B_1)$.
Since~$A$ is visualized, Condition~\ref{CB1} ensures that it only contains entries of~$A_1$ and entries $ij$ with $g<j<i$,
so that the only possible critical cycle that would share a node with~$Z_0=1\cdots g1$ is $Z_1=1\cdots n1$.
Finally, since~$\crit(A)$ is strongly connected and contains $g+1$ and $Z_0$,~$Z_1$ is critical and we have $\digr(A_1)\subset\crit(A)$.
Since~$g$ and~$n$ are coprime, $\gamma(\crit(A))=1$.

{\em If $g=1$}, since~$i_0$ and~$Z_0$ are critical, and $\crit(A)$ is strongly connected, 
there is a walk from~$i_0$ to any node of~$Z_0$ with only critical arcs.
Since~$A$ is strictly visualized, this walk has weight~$0$ and by optimality of $W_0$, $W_0$ also uses only critical arcs from $i_0$ to $Z_0$.
But, by Lemma~\ref{l:interlace}, $W_0$ goes through every node not in~$Z_0$ before reaching~$Z_0$, so that all nodes are critical.

{\em In both cases}, $\crit(A)$ is strongly connected, has cyclicity~$1$ and contains all nodes.
Thus, the strict visualization ensures that $(CS^tR)_{i,j}=\1$ for all~$i,j$, and that $A_{i,j}=\1$ if and only if $(i,j)$ is a critical arc.
Therefore, we can redefine $A_1$ and $A_2$ by $(A_1)_{i,j}=A_{i,j}=\1$ and~$(A_2)_{i,j}=\0$ if $(i,j)$ is critical
and $(A_2)_{i,j}=A_{i,j}<\1$ and~$(A_1)_{i,j}=\0$ otherwise.
We have $A=A_1\oplus A_2$, $\lambda(A_1)=\1$ and $A_2<CSR[A_1]$,
so that Lemma~\ref{l:perturbation} ensures that the index of~$\crit(A)$ is $T(A_1)=T_1(A_1)=T_1(A)=\DM(g,n)$.

\end{proof}

\begin{proof}[Proof of Theorem~\ref{t:CritColWiel}]
Let us first check that the conditions are sufficient : $T_1(A)=T_1(A_1)$ by Lemma~\ref{l:perturbation}
and $T_1(A_1)$ is at least the index of~$\digr(A_1)$ and thus $T_1(A)=\wiel(n)$.
But, since all nodes are critical (because of the Hamiltonian critical cycle),
this means that there is a (critical) row (or column) whose transient is~$\wiel(n)$.

Conversely, let us assume there is a critical row (or column) whose transient is~$\wiel(n)$.
Since $T_1(A)=\wiel(n)$ the existence of~$A_1$ and $A_2$ with $A_2<CSR[A_1]$ follows from Theorem~\ref{t:wiel}
and we already noticed that~$\digr(A_1)$ has index~$\wiel(n)$.
It remains to show the existence of the Hamiltonian cycle.
By Lemma~\ref{l:perturbation}, we know that $\crit(A)=\crit(A_1)$, so that $g=\g(\crit(A))\in \{n-1,n\}.$

If~$g=n-1$, then $\DM(g,n)=\wiel(n)$ and we are in the situation of Theorem~\ref{t:CritColDM}.
Thus, $\crit(A_1)=\crit(A)$ has index~$\wiel(n)$, which is only possible if~$\crit(A)=\digr(A_1)$, that is if $A_1$ has a critical Hamiltonian cycle.

If~$g=n$, then there is an Hamiltonian cycle in~$\crit(A)$,
which is necessarily a critical Hamiltonian cycle of~$\digr(A_1)$ because~$\crit(A_1)=\crit(A)$.
\end{proof}

\section*{Acknowledgement} We are grateful to St\'{e}phane Gaubert and anonymous referees of the paper for many constructive and helpful suggestions and comments.

\providecommand{\bysame}{\leavevmode\hbox
to3em{\hrulefill}\thinspace}
\providecommand{\MR}{\relax\ifhmode\unskip\space\fi MR }
\providecommand{\MRhref}[2]{%
  \href{http://www.ams.org/mathscinet-getitem?mr=#1}{#2}
} \providecommand{\href}[2]{#2}


\begin{thebibliography}{10}


\bibitem{AGW-05}
Marianne Akian, St{\'{e}}phane Gaubert, and Cormac Walsh.
Discrete
max-plus spectral theory. {\em Idempotent Mathematics and Mathematical
Physics\/}
  (G.L. Litvinov and V.P. Maslov, eds.), Contemporary Mathematics, vol. 377,
  AMS, Providence, 2005, pp.~53--77.

\bibitem{BCOQ-92}
Fran\c{c}ois Baccelli, Guy Cohen, Geert Jan Olsder, and Jean-Pierre Quadrat.
{\em Synchronization and Linearity: An Algebra for Discrete Event Systems}.
Wiley, Chichester, 1992.

\bibitem{BH-00}
Fran\c{c}ois Baccelli and Dohy Hong.
TCP is max-plus linear and what it tells us on its throughput.
{\em Proceedings of the Conference on Applications, Technologies, Architectures, and Protocols for Computer Communication (SIGCOMM 2000)}.
ACM, New York, 2000, pp.~219--230.
  
\bibitem{BG-00}
Anne Bouillard and Bruno Gaujal, Coupling time of a (max,plus)
matrix.
{\em Proceedings of the Workshop on Max-Plus Algebra at the 1st IFAC Symposium
	on
System Structure and Control}. Elsevier, Amsterdam, 2001, pp.~335--400.

\bibitem{BR}
Richard~A.~Brualdi and Herbert~J.~Ryser.
\newblock {\em Combinatorial Matrix Theory}.
\newblock Cambridge University Press 1991.

\bibitem{But}
Peter~Butkovi\v{c}. 
\newblock {\em Max-linear Systems: Theory and Algorithms}.
\newblock Springer, London, 2010.

\bibitem{CBFN-17}
Bernadette Charron-Bost, Matthias F{\"{u}}gger, and Thomas Nowak.
New transience bounds for max-plus linear systems. 
{\em Discrete Applied Mathematics\/}~219:83--99, 2017.

\bibitem{CFWW-15}
Bernadette Charron-Bost, Matthias F\"ugger, Jennifer L. Welch, and Josef Widder.
Time complexity of link reversal routing.
{\em ACM Transactions on Algorithms\/}~11(3):18, 2015.


\bibitem{Cohen+}
Guy Cohen, Didier Dubois, Jean-Pierre Quadrat, and Michel Viot,
\emph{Analyse
  du comportement p{\'{e}}riodique de syst{\`{e}}mes de production par la
  th{\'{e}}orie des dio{\"{\i}}des}, INRIA, Rapport de Recherche no.~191,
  Le Chesnay,
  1983.

\bibitem{CMQV-89}
Guy Cohen, Pierre Moller, Jean-Pierre Quadrat, and Michel Viot.
Algebraic tools for the performance evaluation of discrete event systems.
{\em Proceedings of the IEEE\/}~77(1):39--85, 1989.

\bibitem{DM-64}
Anthony~L. Dulmage and Nathan~S. Mendelsohn. 
Gaps in the exponent set of
  primitive matrices. 
{\em Illinois Journal of  Mathematics\/}~8(4):642--656, 1964.

\bibitem{ER-97}
S. Even and S. Rajsbaum.
The use of a synchronizer yields the maximum computation rate in distributed networks.
{\em Theorey of Computing Systems\/}~30:447--474, 1997.

\bibitem{HA-99}
Mark Hartmann and Cristina Arguelles. 
Transience bounds for
long walks.
{\em Mathematics of Operations Research\/}~{24}(2):414--439, 1999.



\bibitem{MNS}
Glenn Merlet, Thomas Nowak, and Serge{{\u\i}} Sergeev.
{Weak CSR expansions and transience bounds in max-plus algebra.}
{\em Linear Algebra and Its Applications\/}~{461}:163--199, 2014.

\bibitem{MNSS}
Glenn Merlet, Thomas Nowak, Serge{{\u\i}} Sergeev, and Hans Schneider.
Generalizations of bounds on the index of convergence to weighted
digraphs.
{\em Discrete Applied Mathematics\/}~{178}:121--134, 2014.

\bibitem{Mol-03}
Monika Moln{\'a}rov{\'a}. {Computational complexity of
{Nachtigall}'s
  representation}. {\em Optimization\/}~{52}:93--104, 2003.

\bibitem{Nacht}
Karl Nachtigall. {Powers of matrices over an extremal algebra
with
  applications to periodic graphs}. 
{\em Mathematical Methods of Operations Research\/}~{46}:87--102, 1997.

\bibitem{SS-91}
Hans Schneider, Michael H. Schneider, Max-balancing weighted directed graphs and matrix scaling,
{\em Mathematics of Operations Research\/}~16:208–222, 1991.
   
\bibitem{Ser-09}
Serge{{\u\i}} Sergeev. {Max algebraic powers of irreducible
matrices in
  the periodic regime: An application of cyclic classes}. 
{\em Linear Algebra and Its Applications\/}:431(6):1325--1339, 2009.

  
\bibitem{SS-11}
Serge{{\u\i}} Sergeev and Hans Schneider. 
{CSR} expansions of
matrix
  powers in max algebra. 
{\em Transactions of the AMS\/}~{364}:5969--5994, 2012.
  
\bibitem{Shao}
Jia-yu Shao. {On the exponent of a primitive digraph}. 
{\em Linear Algebra and Its Applications\/}~64:21--31, 1985.

\bibitem{SyK:03}
Gerardo Soto~y Koelemeijer. \emph{On the behaviour of classes of
min-max-plus
  systems}. Ph.D. thesis, Delft University of Technology, 2003.  

\bibitem{Wie-50}
Helmut Wielandt. 
{Unzerlegbare, nicht negative {M}atrizen}.
{\em Mathematische
Zeitschrift\/}~{52}(1):642--645, 1950.


\end{thebibliography}
\end{document}